\documentclass[11pt]{article}
\usepackage{hyperref}
\usepackage{latexsym,amsmath,amscd,amsthm,amssymb,graphics,mathrsfs,esint}
\usepackage{tikz-cd}
\usepackage{dsfont}
\usepackage[left=2.5cm,right=2.5cm,top=2.5cm]{geometry}
\usepackage{amsthm} 

\newcommand{\Diff}{\operatorname{Diff}}

\newcommand{\Aut}{\operatorname{Aut}}
\newcommand{\lie}[1]{\mathfrak{#1}}
\newcommand{\N}{\mathbb{N}}
\newcommand{\Z}{\mathbb{Z}}
\newcommand{\R}{\mathbb{R}}

\newcommand{\U}{\mathrm{U}}
\newcommand{\cG}{\mathcal{G}}
\newcommand{\cP}{\mathcal{P}}
\newcommand{\cM}{\mathcal{M}}
\newcommand{\X}{\mathfrak{X}}
\newcommand{\ev}{\mathrm{ev}}
\newcommand{\ex}{\mathrm{ex}}
\newcommand{\ol}[1]{\overline{#1}}
\newcommand{\be}{\begin{equation}}
\newcommand{\ee}{\end{equation}}
\newcommand{\curv}{\operatorname{curv}}
\newcommand{\Flux}{\operatorname{Flux}}

\newcommand{\flux}{\operatorname{flux}}
\newcommand{\ham}{\operatorname{ham}}

\newcommand{\Hom}{\operatorname{Hom}}
\newcommand{\one}{\mathbf{1}}
\newcommand{\grb}{\mathcal{G}rb}
\newcommand{\tgress}{\mathscr{T}}
\newcommand{\fusb}{\mathcal{F}us\mathcal{B}un}
\newcommand{\mc}[1]{\mathcal{#1}}
\newcommand{\ph}{\varphi}
\newcommand{\Mor}{\operatorname{Mor}}
\newcommand{\1}{\mathds{1}}
\newcommand{\pr}{\operatorname{pr}}
\newcommand{\fg}{\mathfrak{g}}
\newcommand{\fn}{\mathfrak{n}}
\newcommand{\fz}{\mathfrak{z}}
\newcommand{\dR}{\mathrm{dR}}
\newcommand{\Ham}{\mathrm{Ham}}
\newcommand{\action}{\rho} 
\newcommand{\comp}[1]{\langle #1 \rangle} 
\newcommand{\Ab}{Z}
\newcommand{\ab}{\mathfrak{z}}

\definecolor{darkolivegreen}{rgb}{0.33, 0.42, 0.18} 
\definecolor{cobalt}{rgb}{0.0, 0.24, 0.43}

\makeatletter
\hypersetup{
colorlinks=true,
urlcolor=cobalt,
linkcolor=cobalt,
citecolor=darkolivegreen,
pdfmenubar=false,
pdftoolbar=true,
pdftitle = {\@title},
}
\makeatother

\begin{document}

\theoremstyle{definition}
\newtheorem{definition}{Definition}[section]
\newtheorem{remark}[definition]{Remark}
\newtheorem{example}[definition]{Example}
\newtheorem{article}[definition]{}

\theoremstyle{plain}
\newtheorem{theorem}[definition]{Theorem}
\newtheorem*{theorem*}{Theorem}
\newtheorem{proposition}[definition]{Proposition}
\newtheorem{lemma}[definition]{Lemma}
\newtheorem{corollary}[definition]{Corollary}
\newtheorem{conjecture}[definition]{Conjecture}
\newtheorem{problem}{Problem}

\title{How an action that stabilizes a bundle gerbe gives rise to a Lie group extension}
\author{Bas Janssens\footnote{Institute of Applied Mathematics, Delft University of Technology, Delft, The Netherlands}, Peter Kristel\footnote{Hausdorff Center, University of Bonn, Bonn, Germany}}

\maketitle
\date

\begin{abstract}
Let $\mathcal{G}$ be a bundle gerbe with connection on a smooth manifold $M$, and let $\rho: G \rightarrow \operatorname{Diff}(M)$ be a smooth action of a Fr\'echet--Lie group $G$ on $M$ that preserves the isomorphism class of $\mathcal{G}$.
In this setting, we obtain an abelian extension of $G$ that consists of pairs $(g,A)$, where $g \in G$, and $A$ is an isomorphism from $\rho_{g}^{*}\mathcal{G}$ to $\mathcal{G}$.
We equip this group with a natural structure of abelian Fr\'{e}chet--Lie group extension of $G$, under the assumption that the first integral homology of $M$ is finitely generated.
As an application, we construct the universal central extension (in the category of Fr\'echet--Lie groups) of the group of Hamiltonian diffeomorphisms of a symplectic surface.
As an intermediate step, we obtain a central extension of the group of exact volume-preserving diffeomorphisms of a 3-manifold whose corresponding Lie algebra extension is conjectured to be universal.
\end{abstract}


\tableofcontents


\section{Introduction}\label{sec:introduction}

Let $M$ be a connected smooth manifold. A bundle gerbe with connection is a geometric object on $M$, introduced by Murray \cite{Mu96}, which should be seen as a categorified version of a line bundle with connection.
This analogy runs deep -- many constructions using line bundles have bundle gerbe equivalents.
The present paper is concerned with one such construction, namely the higher Kostant--Kirillov--Souriau (KKS) extension.
We first recall the line bundle version of this extension.

\paragraph{The KKS extension.}
Let $\Ab$ be an abelian Lie group, and let $P$ be a principal $\Ab$-bundle with connection over $M$. 
Let $\rho:G \rightarrow \Diff(M)$ be a smooth action of a Fr\'echet--Lie group $G$ that preserves the isomorphism class of $P$.
For every $g\in G$, there then exists an isomorphism $A \colon \rho_{g}^{*}P \xrightarrow{\sim} P$.
Since $A$ is determined only up to a choice of vertical automorphism, the group 
\[\widehat{G}_{P} := \{(g, A) \,|\, A \colon \rho_{g}^{*}P \xrightarrow{\sim} P\} \]
of isomorphisms that cover the $G$-action 
is a central extension 
of $G$ by the abelian Lie group $H^0(M,\Ab) \simeq \Ab$,
\be
H^0(M,\Ab) \rightarrow \widehat{G}_{P} \rightarrow G.
\ee
 This \emph{KKS-extension} is again a Fr\'echet--Lie group \cite{NV03, DJNV21},
and the corresponding Lie algebra extension 
\be
H_{\mathrm{dR}}^0(M,\R) \rightarrow \widehat{\fg}_{\omega} \rightarrow \fg
\ee
of $\fg$ by $H_{\mathrm{dR}}^0(M,\R) \simeq \R$
depends only on the curvature $\omega \in \Omega^2(M,\ab)$ of $P$.

%
%


\paragraph{KKS extensions for bundle gerbes.}
In the present paper, the main objects of study are \emph{higher KKS-extensions} arising from bundle gerbes with connection. 
 Let $\mc{G}$ be a $\Ab$-bundle gerbe with connection on $M$, and let
 $\rho:G \rightarrow \Diff(M)$ be a smooth action of a Fr\'echet--Lie group on $M$ that preserves the isomorphism class of the bundle gerbe with connection.
 Then, for every $g\in G$, there exists a $1$-isomorphism $A \colon \rho_{g}^{*}\cG \xrightarrow{\sim} \cG$ in the $2$-category of bundle gerbes with connection \cite{Wa07}. 
 Since the equivalence class $\overline{A}$ of $A$ modulo $2$-morphisms is only determined up to the Picard group $H^1(M,\Ab)$, the group 
\be
\widehat{G}_{\mc{G}} := \{(g, \overline{A})\,|\, A \colon \rho_{g}^{*}\mc{G} \xrightarrow{\sim} \mc{G}\}
\ee
of equivalence classes $\overline{A}$ of $1$-isomorphisms that cover the $G$-action is an abelian extension of $G$ by $H^1(M,\Ab)$,
 \be\label{eq:H1extIntro}
H^1(M,\Ab) \rightarrow \widehat{G}_{\mc{G}} \rightarrow G.
\ee
The action of $\widehat{G}_{\mc{G}}$ on $H^1(M,\Ab)$ by conjugation coincides with the pullback along the $G$-action, so, in particular, the extension~\eqref{eq:H1extIntro} is central if $G$ is connected.
We give a detailed construction of this extension in Section~\ref{sec:GroupOfIso}.
\hypertarget{target:A}{In Section~\ref{sec:smoothstructure} we prove (constructively) the following result:}
\begin{theorem*}[A]\label{thm:A}
	If $M$ is an orientable, finite-dimensional manifold for which $H_1(M,\Z)$ is finitely generated, 
	then \eqref{eq:H1extIntro} is an abelian extension of Fr\'{e}chet--Lie groups. 
\end{theorem*}
Moreover, we prove that this structure is natural, in the sense that 1) it does not depend on any auxiliary choices; and 2) it is functorial (see Prop.~\ref{prop:smoothfunctor}).
The corresponding central extension of Fr\'echet--Lie algebras
\begin{equation}\label{eq:AlgebraExtension}
	H^1_{\dR}(M,\ab) \rightarrow \widehat{\fg}_{\omega} \rightarrow \fg 
\end{equation} 
depends only on the curvature $\omega \in \Omega^3(M)$ of $\mc{G}$, and is described explicitly in Section~\ref{sec:smoothstructure}, see also \cite{Ro95, Ne05, DJNV21}.

A central tool that we use to prove Theorem~\hyperref[thm:A]{A} is \emph{transgression to the loop space}, see Section~\ref{subsec:fusion} and~\cite{B93, BM94, Wa16}.
In brief, by transgression, the bundle gerbe $\mc{G}$ gives rise to a principal $\Ab$-bundle with connection over the loop space $LM$ of $M$.
This principal bundle moreover comes equipped with a fusion structure.
Using the functoriality of transgression, we can interpret $\widehat{G}_{\cG}$ as the group of automorphisms of the transgressed principal bundle that cover the natural action of $G$ on $LM$, and which preserve the connection as well as the fusion product, see Proposition \ref{prop:GroupTransgression}.
This allows us to equip $\widehat{G}_{\mc{G}}$ with the structure of Fr\'{e}chet-Lie group.

\paragraph{Universal central extension of the Hamiltonian diffeomorphism group of a surface.}

As an application of Theorem \hyperref[thm:A]{A}, we construct the universal central extension of the group $\Ham(\Sigma)$ of Hamiltonian diffeomorphisms of a closed, connected, symplectic 2-manifold $\Sigma$.

First, in Section~\ref{sec:volumePreserving}, we construct a central extension (in the category of Fr\'{e}chet-Lie groups) of the group $\Diff_{\ex}(M,\mu)$ of exact volume-preserving diffeomorphisms of a closed, orientable 3-manifold $M$.
This is done by applying Theorem \hyperref[thm:A]{A} to a bundle gerbe whose curvature is equal to $\mu$ (up to multiplication by a real constant).
This extension integrates a well-studied central extension of Lie algebras, which is conjectured to be universal \cite{Ro95}.
If this is indeed the case, then Neeb's Recognition Theorem \cite{Neeb2002UCE} provides a condition on the fundamental group of $\Diff_{\ex}(M,\mu)$ that guarantees that (a simply connected cover of) the Lie group extension we constructed is universal.
	
	In Sections~\ref{sec:quant} and~\ref{sec:UniversalHamilton}, we combine this with ideas from \cite{JV18} in order to construct the universal central extension of $\Ham(\Sigma)$.
	We rescale the symplectic form on $\Sigma$ to be prequantizable, and apply Theorem~\hyperref[thm:A]{A}
	to the group $\Aut_{0}(P,\theta)$ of quantomorphisms of the prequantum $\U(1)$-bundle $(P,\theta)$ over~$\Sigma$.
	Using the functoriality of our construction, we exhibit a 1-dimensional embedded Lie subgroup of the central extension 
	$\widehat{\Aut}_{0}(P,\theta)_{\mc{G}}$ that covers the group $\U(1) \subseteq \Aut_{0}(P,\theta)$ of vertical 
	automorphisms.
	The quotient of $\widehat{\Aut}_{0}(P,\theta)_{\mc{G}}$ by this $1$-dimensional embedded subgroup
	yields a central extension of $\Ham(\Sigma)$ whose Lie algebra is universal.
	Using classical results by Smale~\cite{Smale1959}, Earle--Eells~\cite{EarleEells1969} and Gramain~\cite{Gramain1973} on the homotopy type of $\Diff(\Sigma)$, we are able to verify that the condition of Neeb's Recognition Theorem is satisfied for the simply connected cover of the group we constructed.
	This proves that it is indeed the (unique) universal central extension of $\Ham(\Sigma)$ in the category of Fr\'echet--Lie groups.

\paragraph{Higher bundle gerbes.}
KKS extensions for higher bundle gerbes were studied 
at the infinitesimal level (as $L_{\infty}$-algebras) in \cite{FRS14}, and at the global level (as smooth $\infty$-groups) in \cite{BS23}, cf. also \cite{BMS21}
for the case of $1$-bundle gerbes without connection.
Our approach is different in the sense that we discard higher coherences in order to obtain a $1$-group, and we endow this smaller object with the (much finer) structure of a Fr\'{e}chet--Lie group.

This is similar in spirit to the approach in \cite{DJNV21}, where
transgression of differential characters (of arbitrary degree) is used to obtain central extensions of $G$ by $\mathrm{U}(1)$ (which plays the role of $\Ab$), cf.\ also \cite{Ismagilov1996, HV04, DJNV20}.
The main difference with this approach is that we take into account the \emph{fusion structure} on the transgressed bundle.
This extra feature allows us to drop the condition in \cite{DJNV21} that every map $\action_{g} \colon M \rightarrow M$ should be homotopic to the identity. 
Moreover, it enables us to find a single extension by $H^1(M,\Ab)$, as opposed to $\mathrm{Rk}(H_1(M,\Z))$ many extensions by~$\Ab$.
The prominence of the fusion structure in our construction also explains why it only works for $1$-gerbes.

\paragraph{Acknowledgements}  
B.J.~is supported by the NWO grant 639.032.734
`Cohomology and representation theory of infinite dimensional Lie groups'.
P.K.~is supported by a postdoctoral fellowship from the Hausdorff Center for Mathematics.

\section{Preliminaries: Bundle gerbes and transgression}

Bundle gerbes were introduced in \cite{Mu96} as geometric realizations of classes in $H^{3}(M,\Z)$.
In this sense, they are higher analogues of line bundles, which realize classes in $H^{2}(M,\Z)$.
Given an $n$-form $\omega$ on $M$, one obtains an $n-1$ form on $LM$ by integrating $\mathrm{ev}^{*}\omega \in \Omega^{n}(S^{1} \times LM)$ along $S^{1}$, where $\mathrm{ev}: S^{1} \times LM \rightarrow M$ is the evaluation map.
This induces a map $H^{n}(M,\Z) \rightarrow H^{n-1}(LM,\Z)$, which is called \emph{transgression}.
It turns out that in the case $n=3$, this operation lifts to the geometric realizations, i.e.~one can transgress a bundle gerbe on $M$ to obtain a line bundle on $LM$.
In this section, we briefly recall the definition of the strict monoidal 2-groupoid of bundle gerbes with connection from \cite{Wa07}, and the associated transgression-regression machinery as set up in \cite{Wa16}.

\subsection{The strict monoidal 2-groupoid of bundle gerbes with connection}\label{sec:2catBgb}


Let $M$ be a finite-dimensional manifold, and let $\Ab$ be a finite-dimensional abelian Lie group with Lie algebra $\ab$.
Even though the base space $M$ is a manifold, we must work with \emph{diffeological} bundle gerbes in order to properly work with transgression.
This does no harm, because the 2-category of diffeological bundle gerbes with connection over $M$ is equivalent to the 2-category of ordinary bundle gerbes with connection, \cite[Theorem 3.1.7]{Wa16}.

Given a subduction $\pi: Y \rightarrow M$ , we write $Y^{[k]}$ for the $k$-fold fibre product of $Y$ over $M$.
We write $\pr_{23...k}: Y^{[k]} \rightarrow Y^{[k-1]}$ for the projection map which forgets the first component, and similarly for $\pr_{13...k}$ and so on.

\begin{definition}
 A \emph{bundle gerbe} on $M$ consists of the following data:\vspace{-.4em}
\begin{itemize}\itemsep-.2em
	\item A subduction $\pi: Y \rightarrow M$,
	\item A principal $\Ab$-bundle $P \rightarrow Y^{[2]}$.
	\item An isomorphism $\lambda:\pr_{12}^{*}P \otimes \pr_{23}^{*}P \rightarrow \pr_{13}^{*}P$ of principal $\Ab$-bundles over $Y^{[3]}$.
\end{itemize}
The isomorphism $\lambda$ is required to be associative over $Y^{[4]}$ as explained in, e.g., \cite[Def.~2.1.1]{Wa16}.
\end{definition}

\begin{definition}\label{def:ConnectionAndCurvature}
A \emph{connection} on a bundle gerbe $\mc{G} = (Y, P, \lambda)$ consists of:
\begin{itemize}\itemsep-.2em
 \item A principal connection on $P$, compatible with $\lambda$.
 \item A 2-form $B \in \Omega^{2}(Y, \ab)$ such that $\mathrm{curv}(P) = \pr_{2}^{*}B - \pr_{1}^{*}B$.
\end{itemize}
	The 2-form $B$ is called the \emph{curving} of the connection.
	There exists a unique 3-form $\omega \in \Omega^{3}(M,\ab)$, which satisfies $dB = \pi^{*}\omega$.
	This 3-form is called the \emph{curvature} of the connection.
	If $\omega = 0$, then the bundle gerbe is called \emph{flat}.
\end{definition}
Every bundle gerbe with connection $\mc{G} = (Y,P,\lambda,B)$ admits a \emph{dual} $\mc{G}^{*} := (Y,P^{*},(\lambda^{*})^{-1},-B)$, see \cite{Wa07}.
\begin{definition}\label{def:MorphismsOfBundleGerbes}
Given two bundle gerbes with connection, $\mc{G}_{1} = (Y_{1},P_{1},\lambda_{1},B_{1})$ and $\mc{G}_{2}=(Y_{2},P_{2},\lambda_{2},B_{2})$, an 
\emph{invertible 1-morphism} $A$ from $\mc{G}_{1}$ to $\mc{G}_{2}$ consists of the following data:
\begin{itemize}\itemsep-.2em
	\item A subduction $\zeta: \mc{Y} \rightarrow Y_{1} \times_{M} Y_{2}$,
	\item A principal $\Ab$-bundle $A \rightarrow \mc{Y}$ with connection,
	\item An isomorphism $\alpha: P_{1} \otimes \zeta_{2}^{*}A \rightarrow \zeta_{1}^{*}A \otimes P_{2}$ of principal $\Ab$-bundles.
\end{itemize}
The connection $A$ must be compatible with the curvings $B_{1}$ and $B_{2}$ in the sense that
\begin{equation}\label{eq:CurvatureAndCurvings}
	\curv(A) = \pr_{2}^{*}B_{2} - \pr_{1}^{*}B_{1} \in \Omega^{2}(\mc{Y},\ab),
\end{equation}
where $\mathrm{pr}_{i}$ is the obvious projection from $\mc{Y}$ to $Y_{i}$.
The isomorphism $\alpha$ is required to be compatible with the isomorphisms $\lambda_{1}$ and $\lambda_{2}$ in the sense of \cite[Def.~2, 1M2]{Wa07}.
\end{definition}

There are further notions of tensor products and of 2-morphisms, this leads to a monoidal 2-groupoid called the \emph{2-groupoid of bundle gerbes with connection}, denoted $\grb(M)$, see \cite[Lemma 3.1.6]{Wa16} and \cite[Section 2]{Wa07}.
For brevity, we will take ``bundle gerbe'' to mean ``bundle gerbe with connection''.
In this article, our interest will be in the 1-truncation of this 2-category denoted by $h_{1}\grb(M)$.
The objects of $h_{1}\grb(M)$ are those of $\grb(M)$, and its morphisms are 2-isomorphism classes of 1-morphisms in $\grb(M)$.

To every smooth map $f: M \rightarrow N$ one can associate a ``pullback'' operation, $f^{*}: \grb(N) \rightarrow \grb(M)$.
The operation $f^{*}$ is a strict (monoidal) 2-functor, \cite[Section 1.4]{Wa07}.
It follows that $f^{*}$ induces a functor $f^{*}: h_{1}\grb(N) \rightarrow h_{1}\grb(M)$.
One particular consequence of functoriality is that, if $A: \mc{G} \rightarrow \mc{H}$ and $B: \mc{H} \rightarrow \mc{K}$ are morphisms between objects of $h_{1}\grb(N)$, then
\begin{equation}\label{eq:PullbackFunctor}
	f^{*}(B \circ A) = f^{*}(B) \circ f^{*}(A).
\end{equation}
Moreover, if $\1_{\mc{G}}$ is the identity morphism on $\mc{G}$, then
\begin{equation}\label{eq:PullbackIdentity}
	f^{*}\1_{\mc{G}} = \1_{f^{*}\mc{G}}.
\end{equation}

\subsection{Holonomy}
There is a notion of holonomy for bundle gerbes with connection (see, e.g.~\cite{Wa16}).
Let $\mc{G} = (Y,P, \lambda,B)$ be a bundle gerbe with connection.
\begin{definition}
	Let $\Sigma$ be a closed, oriented surface, and let $\sigma: \Sigma \rightarrow M$ be a smooth map.
	The bundle gerbe $\sigma^{*}\mc{G}$ defines a class $[\sigma^{*}\mc{G}] \in H^{2}(\Sigma,\Ab)$.
	The \emph{holonomy} of $\mc{G}$ around $\Sigma$ is the pairing
	\begin{equation*}
		\mathrm{hol}_{\mc{G}}(\Sigma,\sigma) := \langle [\sigma^{*}\mc{G}],[\Sigma] \rangle,
	\end{equation*}
	where $[\Sigma] \in H_{2}(\Sigma, \Z)$ is the fundamental class.
\end{definition}
For a smooth manifold $X$, we write $\mc{I}_{\rho}(X)$ for the trivial bundle gerbe with connection whose curving is $\rho \in \Omega^{2}(X,\ab)$.
Let $\Sigma$ be an oriented surface, and let $\sigma: \Sigma \rightarrow M$ be a smooth map.
Suppose that $\sigma^{*}\mc{G}$ is isomorphic to $\mc{I}_{\rho}(\Sigma)$.
We then have
\begin{equation*}
	\mathrm{hol}_{\mc{G}}(\Sigma,\sigma) = \exp \int_{\Sigma} \rho.
\end{equation*}
We should point out that if $\Ab$ is not connected, then $\sigma^{*}\mc{G}$ need not be trivializable.

\begin{lemma}\label{Lemma:HolonomyCurvature}
	Let $\phi: \R \rightarrow \Diff(M)$ be smooth curve.\footnote{This means that the map $\R \times M \rightarrow M, \;(t, x) \mapsto \phi_t(x)$ is smooth.}
	Let $\Sigma$ be a closed, oriented surface, and let ${\sigma: \Sigma \rightarrow M}$ be a smooth  map.
	Define $\sigma_t \colon \Sigma \rightarrow M$ by $\sigma_t := \phi_t \circ \sigma$, and define 
	$\Phi \colon \Sigma \times \R \rightarrow M$ by $\Phi(s,t) := \sigma_t(s)$.
	Let $\omega \in \Omega^{3}(M,\ab)$ be the curvature of $\mc{G}$.
	We then have
	\begin{equation}\label{eq:HolonomyVSCurvature}
		\mathrm{hol}_{\mc{G}}(\Sigma, \sigma_{T}) - \mathrm{hol}_{\mc{G}}(\Sigma, \sigma_{0}) = \exp \int_{\Sigma \times [0,T]} \Phi^{*}\omega,
	\end{equation}
	for all $T \in \R$.
\end{lemma}

Let $p_{\Sigma} \colon \Sigma \times \R \rightarrow \Sigma$ be the projection onto the first factor, and let $i_{t}: \Sigma \rightarrow \Sigma \times \R, s \mapsto (s,t)$ be inclusion at time $t$.
\begin{lemma}
	The bundle gerbe $p^{*}_{\Sigma}i_{0}^{*}\Phi^{*}\mc{G}$ is isomorphic to $\Phi^{*}\mc{G}$, as a bundle gerbe without connection.
\end{lemma}
We moreover remark that $p^{*}_{\Sigma}i_{0}^{*}\Phi^{*}\mc{G}$ is flat, since $i_{0}^{*}\Phi^{*}\mc{G}$ is flat (for dimensional reasons).

\begin{proof}
	Recall that bundle gerbes without connection on a manifold $N$ are classified by the sheaf cohomology in degree 2 of the sheaf of smooth functions with values in $\Ab$, denoted $\check{H}^2(N,\underline{\Ab})$.
	The short exact sequence of Lie groups
	\begin{equation*}
		\pi_1(\Ab_0) \rightarrow \mathrm{Lie}(\Ab) \rightarrow \Ab_0
	\end{equation*}
	gives rise to a corresponding short exact sequence of sheaves.
	Since the sheaf $\underline{\mathrm{Lie}(\Ab)}$ of smooth functions with values in $\mathrm{Lie}(\Ab)$ is acyclic, 
	one infers from the corresponding long exact sequence in sheaf cohomology that 
	\be \label{eq:SheafCohomologyConnected}\check{H}^k(N, \underline{\Ab_0}) \simeq H^{k+1}(N, \pi_1(\Ab_0)).\ee
	(The sheaf cohomology reduces to singular cohomology because $\pi_1(\Ab_0)$ is discrete.)
	In the same vein, the short exact sequence of Lie groups 
	\[
	\Ab_0 \rightarrow \Ab \rightarrow \pi_0(\Ab)
	\]
	gives rise to a long exact sequence in sheaf cohomology, which reduces to
	\be\label{eq:LESsheafcohomology}
	\ldots \rightarrow H^{k+1}(N, \pi_1(\Ab_0)) \rightarrow \check{H}^k(N, \underline{\Ab}) \rightarrow H^k(N, \pi_0(\Ab)) \rightarrow H^{k+2}(N, \pi_1(\Ab_0)) \rightarrow \ldots .
	\ee
	if one substitutes \eqref{eq:SheafCohomologyConnected}.
	
	Note that both $N = \Sigma$ and $N = \Sigma \times \R$ satisfy $H_{k}(N,\Z) = 0$ for $k >2$, so $H^k(N,\pi_1(\Ab_0)) = 0$ for $k >2$
	by the universal coefficient theorem and the fact that $\pi_1(\Ab_0)$ is a free $\Z$-module.
	It then follows from \eqref{eq:LESsheafcohomology} that 
	\begin{equation}\label{eq:CohomologyIsos}
		\check{H}^2(\Sigma, \Ab) \simeq H^2(\Sigma, \pi_0(\Ab))\qquad \text{ and } \qquad\check{H}^2(\Sigma\times \R, \Ab) \simeq H^2(\Sigma\times \R, \pi_0(\Ab)).
	\end{equation}
	Since $i_{0} p_{\Sigma}$ is homotopic to the identity (and $p_{\Sigma} i_{0}$ is the identity) we have that $p_{\Sigma}^{*}: H^{2}(\Sigma, \pi_{0}(\Ab)) \rightarrow H^{2}(\Sigma \times \R, \pi_{0}(\Ab))$ is an isomorphism, with inverse $i_{0}^{*}$.
	Because the isomorphisms in Equation \eqref{eq:CohomologyIsos} are compatible with pullback, we have that $p^{*}_{\Sigma}: \check{H}^{2}(\Sigma,\Ab) \rightarrow \check{H}^{2}(\Sigma \times \R, \Ab)$ is also an isomorphism, with inverse $i_{0}^{*}$.
	Thus, $p^{*}_{\Sigma} i_{0}^{*}\Phi^{*}\mc{G}$ is isomorphic to $\Phi^{*}\mc{G}$, as a bundle gerbe without connection.
\end{proof}

\begin{proof}[Proof of Lemma~\ref{Lemma:HolonomyCurvature}]
Let $T \in \R$.
First, we prove that if $\mc{G}' \rightarrow \Sigma \times \R$ is a trivializable bundle gerbe with curvature $\omega'$, then
\begin{equation*}
	\mathrm{hol}_{\mc{G}'}(\Sigma, i_{T}) - \mathrm{hol}_{\mc{G}'}(\Sigma, i_{0}) = \exp \int_{\Sigma \times [0,T]} \omega'.
\end{equation*}
Let $\pi':Y'\rightarrow \Sigma \times \R$ and $B' \in \Omega^{2}(Y', \ab)$ be the subduction and the curving of $\mc{G}'$.
Trivializability of $\mc{G}'$ means that there exists a 2-form $\rho \in \Omega^{2}(\Sigma \times \R, \ab)$ such that $\mc{G}'$ is isomorphic to $\mc{I}_{\rho}(\Sigma \times \R)$.
Since $i_{t}^{*}\mc{G}'$ is isomorphic to $\mc{I}_{i_{t}^{*}\rho}(\Sigma)$, this implies that $\mathrm{hol}_{\mc{G}'}(\Sigma, i_{t}) = \exp \int_{\Sigma \times \{t\}} \rho$.
Let $A \rightarrow Y'$ be a principal $\Ab$-bundle that is part of an isomorphism between $\mc{G}'$ and $\mc{I}_{\rho}(\Sigma \times \R)$, and let $F_{A} \in \Omega^{2}(Y', \ab)$ be its curvature.
By Definition \ref{def:MorphismsOfBundleGerbes}, in particular Equation \eqref{eq:CurvatureAndCurvings}, we have $F_{A} = (\pi')^{*}\rho - B'$.
Differentiating this equation, we obtain $dB' = (\pi')^{*}d \rho$, whence $d\rho = \omega'$.
We now compute
	\begin{align*}
		\mathrm{hol}_{\mc{G}'}(\Sigma, i_{T}) - \mathrm{hol}_{\mc{G}'}(\Sigma, i_{0}) &= \exp\left( \int_{\Sigma \times \{ T \} } \rho - \int_{\Sigma \times \{ 0 \} }\rho \right) \\
		&= \exp \left( \int_{\Sigma \times [0,T]} d \rho \right) \\
		&= \exp \left( \int_{\Sigma \times [0,T]} \omega' \right).
	\end{align*}

	We now consider again the original bundle gerbe $\mc{G} \rightarrow M$.
	We set $\mc{G}_{0} = p_{\Sigma}^{*}i_{0}^{*}\Phi^{*}\mc{G}$, which is flat and, as bundle gerbe without connection, isomorphic to $\Phi^{*}\mc{G}$ (by Lemma \ref{Lemma:HolonomyCurvature}).
	We claim that $\mathrm{hol}_{\mc{G}_{0}}(\Sigma,i_{t})$ does not depend on $t$, so that in particular
	\begin{equation*}
		\mathrm{hol}_{\mc{G}_{0}} (\Sigma, i_{T}) - \mathrm{hol}_{\mc{G}_{0}}(\Sigma, i_{0}) = 0.
	\end{equation*}
	First, because the curvature of $\mc{G}_{0}$ vanishes, it defines a class $[\mc{G}_{0}] \in H^{2}(\Sigma \times \R,\Ab)$.
	It follows that $\mathrm{hol}_{\mc{G}_{0}}(\Sigma, i_{t}) = \langle [i_{t}^{*}\mc{G}_{0}],[\Sigma] \rangle = \langle i_{t}^{*}[\mc{G}_{0}],[\Sigma] \rangle$, which does not depend on $t$. 
	Indeed, since $i_{t}$ is homotopic to $i_{t'}$ for all $t,t' \in \R$, we have $i_{t}^{*} = i_{t'}^{*}$ as maps from $H^{2}(\Sigma \times \R,\Ab)$ to $H^{2}(\Sigma, \Ab)$.

	The bundle gerbe $\mc{G}' = \Phi^{*}\mc{G} \otimes \mc{G}_{0}^{*}$ is trivializable, and has curvature $\Phi^{*}\omega$.
	It thus follows that
	\begin{equation*}
		\mathrm{hol}_{\mc{G}'}(\Sigma, i_{T}) - \mathrm{hol}_{\mc{G}'}(\Sigma,i_{0}) = \exp \int_{\Sigma \times [0,T]} \Phi^{*}\omega.
	\end{equation*}
	On the other hand, we have
	\begin{align*}
		\mathrm{hol}_{\mc{G}'}(\Sigma, i_{T}) - \mathrm{hol}_{\mc{G}'}(\Sigma, i_{0}) &= \mathrm{hol}_{\Phi^{*}\mc{G} \otimes \mc{G}_{0}^{*}}(\Sigma, i_{T}) - \mathrm{hol}_{\Phi^{*}\mc{G} \otimes \mc{G}_{0}^{*}}(\Sigma, i_{0}) \\
		&= \mathrm{hol}_{\Phi^{*}\mc{G}}(\Sigma, i_{T}) - \mathrm{hol}_{\mc{G}_{0}} (\Sigma, i_{T}) + \mathrm{hol}_{\mc{G}_{0}}(\Sigma, i_{0})- \mathrm{hol}_{\Phi^{*}\mc{G}}(\Sigma, i_{0}) \\
		&= \mathrm{hol}_{\mc{G}}(\Sigma, \sigma_{T}) - \mathrm{hol}_{\mc{G}}(\Sigma, \sigma_{0}),
	\end{align*}
	Where we have used
		\begin{equation*}
			\mathrm{hol}_{\mc{G} \otimes \mc{H}}(\Sigma,\sigma) = \mathrm{hol}_{\mc{G}}(\Sigma,\sigma) + \mathrm{hol}_{\mc{H}}(\Sigma,\sigma)
		\end{equation*}
		and
		\begin{equation*}
			\mathrm{hol}_{\mc{G}^{*}}(\Sigma, \sigma) = -\mathrm{hol}_{\mc{G}}(\Sigma, \sigma).\qedhere
		\end{equation*}
\end{proof}

\subsection{Transgression of bundle gerbes to the loop space}\label{subsec:fusion}

Associated to any bundle gerbe $\mc{G}$ over $M$, there is a principal $\Ab$-bundle $\tgress(\mc{G}) \rightarrow LM$.
Its fibre over a loop $\gamma \in LM$ is
\begin{equation*}
	\tgress(\mc{G})_{\gamma} = h_{0}\Mor(\gamma^{*}\mc{G},\mc{I}),
\end{equation*}
the set of equivalence classes of $1$-morphisms from $\gamma^{*}\mc{G}$ to the trivial bundle gerbe $\mc{I}$ over $S^{1}$.
If $z \in \Ab$, then there is a $\Ab$-bundle with holonomy $z$ on $S^{1}$, which we denote by $P_{z}$.
The fibre $\tgress(\mc{G})_{\gamma}$ is equipped with the $\Ab$-action $(p,z) \mapsto p \otimes P_{z}$.
Moreover, a connection on $\mc{G}$ gives rise to a (superficial and symmetrizing) connection on $\tgress(\mc{G})$ as explained in \cite[Sec.~4.3]{Wa16}.
If $A$ is a 1-morphism from $\mc{G}_{1}$ to $\mc{G}_{2}$, then we obtain a morphism $\tgress(A):\tgress(\mc{G}_{1}) \rightarrow \tgress(\mc{G}_{2})$, which is fibrewise given by the map
\begin{equation}\label{eq:morphismsTransgression}
	\tgress(A): h_{0} \Mor(\gamma^{*}\mc{G}_{1},\mc{I}) \rightarrow h_{0}\Mor(\gamma^{*}\mc{G}_{2},\mc{I}), \quad
	p \mapsto p \circ \gamma^{*}A^{-1}.
\end{equation}

In addition, the product $\lambda$, which is part of the bundle gerbe $\mc{G} = (Y, P, \lambda)$ gives rise to a \emph{fusion structure} on $\tgress(\mc{G}) \rightarrow LM$ (\cite[Theorem 6.2.3, 6.2.4]{B93}, \cite[Theorem 4.4]{BM94}, \cite[\S4]{Wa16}).

To define fusion structures on a principal $\Ab$-bundle $\cP \rightarrow LM$, let $PM$ be the space of paths in $M$ with sitting instants.
We consider it as a diffeological fibre bundle
$\pi \colon PM\rightarrow M \times M$ over $M\times M$, where the bundle map $\pi$ assigns to a path $\gamma \colon [0,1] \rightarrow M$ its start- and end-point: $\pi(\gamma) = (\gamma(0),\gamma(1))$.
We denote its $n$-fold fibre product by $PM^{[n]}$; it is the space of $n$-tuples $(\gamma_1, \ldots, \gamma_{n})$ of paths which start and end at the same two points
$x = \gamma_i(0)$ and $y = \gamma_{i}(1)$.
We denote by $\pr_{i_1,\ldots, i_n} \colon  PM^{[n+k]} \rightarrow PM^{[n]}$ the map which remembers the strands $\gamma_{i_1}$, \ldots, $\gamma_{i_n}$.
The space $PM^{[2]}$ maps to $LM$ by $(\gamma_1,\gamma_2) \mapsto \ol{\gamma}_2*\gamma_1$, 
denote this map by $l \colon PM \rightarrow LM$.
We will often need the pullback  $l^*\cP \rightarrow PM^{[2]}$ of $\cP \rightarrow LM$ 
along this map.
If $M$ is a manifold, then $\cP \rightarrow LM$ is a bundle of Fr\'echet manifods, whereas
$l^*\cP \rightarrow PM^{[2]}$ is merely a diffeological bundle.
We will often write $\cP \rightarrow PM^{[2]}$ instead of $l^*\cP \rightarrow PM^{[2]}$.

\begin{definition}[Fusion structures]
	A \emph{product} on $(\cP,\Theta)$ is a connection-preserving isomorphism
	\[
		\mu \colon \pr_{12}^*\cP \otimes \pr_{23}^*\cP \rightarrow \pr_{13}^*\cP
	\]
	over $P^{[3]}M$.
	A \emph{fusion structure} is a product which is associative, in the sense that
	\be\label{FusionAssociator}
	\mu_{124}(p_{12}\otimes \mu_{234}(p_{23}\otimes p_{34})) = \mu_{134}(\mu_{123}(p_{12}\otimes p_{23})\otimes p_{34})
	\ee
	for all $p_{ij} \in \pr_{ij}^*\cP$ over $P^{[4]}M$, where $\mu_{ijk}$ denotes the isomorphism
	$\pr_{ijk}^*\mu \colon \pr_{ij}^*\cP\otimes \pr_{jk}^*\cP \rightarrow \pr_{ik}^*\cP$ over $P^{[4]}M$.
\end{definition}

Let $\fusb^{\nabla_{\text{sf}}}(LM)$ be the category of fusive principal $\Ab$-bundles with flat, superficial connection on $LM$ as defined in \cite{Wa16}.
The assignment $\tgress:h_{1}\grb(M) \rightarrow \fusb^{\nabla_{\text{sf}}}(LM)$ is then a functor, called \emph{transgression}, (see \cite[Section 4.1]{Wa16}).
In fact, $\tgress$ is an equivalence of categories; this is proved in \cite{Wa16}, where Waldorf constructs a ``regression functor'', which is a weak inverse of transgression.
To be precise, we have the following result.
\begin{theorem}[{\cite[Theorem A]{Wa16}}]
	Let $M$ be a connected smooth manifold with base point $x \in M$.
	Then transgression and regression form an equivalence of monoidal categories,
	\begin{equation*}
		\begin{tikzcd}
			h_{1} \grb(M) \ar[r, shift left, "\tgress"] & \fusb^{\nabla_{\text{sf}}}(LM) \ar[l, shift left, "\mathscr{R}_{x}"].
		\end{tikzcd}
	\end{equation*}
	Moreover, this equivalence is natural with respect to base-point preserving maps.
\end{theorem}

We note that (unlike regression) transgression is natural with respect to all smooth maps $\ph \colon M \rightarrow M$, not just base-point preserving ones, see \cite[Equation (3.4)]{Wa10}.
In the sequel, we shall make use of the following observation.

\begin{proposition}\label{prop:transgressionAndPullback}
We can identify $(L\ph)^{*}\tgress(\mc{G}) = \tgress(\ph^{*}\mc{G})$, and 
$(L\ph)_*\tgress(\ph^{*}A) = \tgress(A) (L\ph)_{*}$.
\end{proposition}
\begin{proof}
This is an immediate consequence of the naturality of $\tgress$ with respect to $\ph\colon M \rightarrow M$.
For the first equality, note that 
the fibre $(L\ph)^{*}\tgress(\mc{G})_{\gamma}$ at the loop $\gamma \in LM$ is equal to 
\[\tgress(\cG)_{\ph \circ \gamma} = h_{0}\Mor((\ph \circ \gamma)^{*}\mc{G},\mc{I}) = 
h_{0}\Mor(\gamma^* (\ph^{*}\mc{G}),\mc{I}) = \tgress(\ph^*\cG)_{\gamma},
\]
which is the fibre of $\tgress(\ph^*\cG)$ over $\gamma$. 
For the second equality, let $p_{\gamma}$ be an element of $(L\ph)^*\tgress(\mc{G})_{\gamma}$. 
Then $(L\ph)_*p_{\gamma}$ is the same morphism $p_{\gamma} \colon (\ph \circ \gamma)^*\cG \rightarrow \mc{I}$, but now considered as an element of 
$\tgress(\cG)_{\ph\circ\gamma}$.
For an isomorphism $A \colon \mc{G} \rightarrow \mc{H}$, $\tgress(A) (L\ph)_{*}p_{\gamma}$ is therefore $p_{\gamma} \circ (\ph \circ \gamma)^*A^{-1}$, considered as an element of $\tgress(\mc{H})_{\ph\circ\gamma}$.
On the other hand, if we consider $p_{\gamma} \colon \gamma^*(\ph^*\cG) \rightarrow \mc{I}$ as an element of $\tgress(\ph^*\cG)_{\gamma}$, 
then $\tgress(\ph^*A)p_{\gamma}$ is equal to $p_{\gamma} \circ \gamma^*(\ph^*A)^{-1}$ as an element of $\tgress(\ph^*\mc{H})_{\gamma}$, and $L\ph^*\tgress(\ph^*A)p_{\gamma}$ is the same thing considered as an element of $\tgress(\mc{H})_{\ph \circ \gamma}$.
\end{proof}

\section{The group of isomorphisms of a gerbe}\label{sec:GroupOfIso}


Let 
$\cG$ be a bundle gerbe with connection on $M$, and let 
$G$ be a group with an action $\action \colon G \times M \rightarrow M$ that is smooth in the second argument.
Then, each $g\in G$ induces a diffeomorphism $\action_{g}: M \rightarrow M$ by $ \action_{g}(x) := \action(g,x)$.
We say that the action preserves the isomorphism class $[\cG]$ of $\cG$ if $\action_{g}^*\cG$ is isomorphic to $\cG$
for all $g\in G$. 
In this setting, we construct an abelian extension $\widehat{G}_{\cG}$ of the group $G$.

Under the additional 
conditions that $G$ is a Fr\'echet--Lie group 
and that the action is smooth in both variables, we will show in \S\ref{sec:smoothstructure} that $\widehat{G}_{\cG}$ is a Fr\'{e}chet-Lie group in a natural way.

\subsection{The abelian extension as a group.}
As a set, the abelian extension $\widehat{G}_{\cG}$ is given by
\be
	\widehat{G}_{\cG} := \Big\{(g, A)\, ; \, g \in G \text{ and } A \in h_1 \Mor (\action_{g}^*\cG, \cG)\Big\}.
\ee
(For brevity, we write $A$, instead of $\ol{A}$, for classes of $1$-morphisms modulo $2$-morphisms.)
The product of two elements $(g, A)$ and $(h, B)$ of $\widehat{G}_{\mc{G}}$ is defined by
\begin{equation}\label{eq:Product}
	(g, A) (h, B) = (gh, B \circ \action_{h}^{*}A).
\end{equation}
The following diagram shows that $B \circ \rho_{h}^{*}A$ is indeed a morphism from $\rho_{gh}^{*}\mc{G}$ to $\mc{G}$:
	\begin{equation*}
		\begin{tikzcd}
			\mc{G} & \ar[l,"B"'] \action_{h}^{*}\mc{G} & \ar[l,"\action_{h}^{*}A"'] \ar[ll,bend left,"B \circ \action_{h}^{*}A"] \action_{gh}^{*}\mc{G}.
		\end{tikzcd}
	\end{equation*}
\begin{proposition}\label{prop:GhatIsGroup}
	When equipped with the binary operation from Equation \eqref{eq:Product}, $\widehat{G}_{\mc{G}}$ is a group.
\end{proposition}
\begin{proof}
	Associativity of the product is a straightforward consequence of Equation~\eqref{eq:PullbackFunctor}: 
		\begin{align*}
			\big( (g,A) (h, B) \big) (f, C) &= (gh, B \circ \action_{h}^{*}A)(f, C) \\
								&= (ghf, C \circ \action_{f}^{*} (B \circ \action_{h}^{*} A)) \\
								&= (ghf, C \circ \action_{f}^{*}B \circ \action_{hf}^{*}A) \\
								&= (g, A)\big( (h, B)(f, C) \big).
		\end{align*}
	The unit of $\widehat{G}_{\mc{G}}$ is $(\1,\1_{\mc{G}})$, where $\1$ is the unit of $G$, and $\1_{\mc{G}}$ is the identity morphism from $\mc{G}$ to $\mc{G}$.
	The inverse of $(g, A)$ is given by the formula
	\begin{equation}
		(g, A)^{-1} = (g^{-1},\action_{g^{-1}}^{*}A^{-1}),
	\end{equation}
	as is easily verified using equations \eqref{eq:PullbackFunctor} and \eqref{eq:PullbackIdentity}:
		\begin{align*}
			(g, A)(g^{-1},\action_{g^{-1}}^{*}A^{-1}) &= (\1, \action_{g^{-1}}^{*}A^{-1} \circ \action_{g^{-1}}^{*}A) \\
							   &= (\1, \action_{g^{-1}}^{*} \1_{\action_{g}^{*}\mc{G}})) = (\1, \1_{\mc{G}}). \qedhere
		\end{align*}
\end{proof}
Now, it is clear that the map $\widehat{G}_{\mc{G}} \rightarrow G, (g, A) \mapsto g$ is a surjective group homomorphism, whose kernel, denoted $K$, consists of the (isomorphism classes of) $1$-isomorphisms from $\mc{G}$ to $\mc{G}$.
We thus obtain an abelian extension
\begin{equation}\label{eq:abextension}
	K \rightarrow \widehat{G}_{\mc{G}} \rightarrow G.
\end{equation}

\begin{remark}
If we take for $G$ the group
\[
	\Diff(M,[\mc{G}]) := \{ \ph \in \Diff(M) \,;\, [\mc{G}] = [\ph^{*}\mc{G}] \} \\
\]
of diffeomorphisms of $M$ that preserve the isomorphism class $[\cG]$ of $\cG$, then $\widehat{G}_{\cG}$ is the group
of symmetries of the bundle gerbe $\cG$,
\[
	\widehat{\Diff}(M,\cG) := \{(\ph, A)\,;\, \ph \in \Diff(M, [\cG]), A \in h_1\Mor(\ph^*\cG, \cG)\}.
\]
For any other group $G$, the abelian extension \eqref{eq:abextension} arises 
by pullback along the group homomorphism $G \rightarrow \Diff(M,[\mc{G}])$;
\begin{equation}\label{diag:pullbackabext}
	\begin{tikzcd}
		\widehat{G}_{\cG} \ar[r,dashed] \ar[d,dashed] & \widehat{\Diff}(M,\cG) \ar[d] \\
		G \ar[r] & \Diff(M, [\cG]).
	\end{tikzcd}
\end{equation}
\end{remark}

\begin{proposition}\label{prop:TheKernel}
	There is a natural isomorphism of groups $K \cong H^{1}(M,\Ab)$.
\end{proposition}
\begin{proof}
	The set of equivalence classes of 1-automorphisms of $\mc{G}$ is a torsor over the group $\mathrm{Pic}_0(M)$
	of isomorphism classes of flat principal $\Ab$-bundles over $M$ by \cite[Corollary 2.5.5]{WaldorfPHD}.
	This is isomorphic to $\mathrm{Hom}(\pi_1(M), \Ab)$ (a flat bundle is completely determined by its holonomy), and
	since $\Ab$ is abelian we have $\Hom(\pi_1(M), \Ab) = \Hom(H_1(M,\Z), \Ab)$. This in turn is canonically isomorphic to 
	$H^1(M,\Ab)$ by the universal coefficient theorem (the $\mathrm{Ext}^{1}_{\Z}(H_0(M,\Z), \Ab)$ term vanishes because $H_0(M,\Z)$ is free).
	
	The torsor map
	\begin{equation*}
		H^{1}(M,\Ab) \times \Aut(\mc{G}) \rightarrow \Aut(\mc{G})
	\end{equation*}
	is given by
	\begin{equation*}
		(L, \mc{A}) \mapsto f(L) \otimes \mc{A},
	\end{equation*}
	where $f$ is the inverse of the map $\Aut(\mc{I}_{0}) \rightarrow H^{1}(M, \Ab)$, defined in \cite{WaldorfPHD}. (Here $\mc{I}_{0}$ is the trival bundle gerbe on $M$.)
	Now, we claim that the map
	\begin{equation*}
		H^{1}(M,\Ab) \rightarrow \Aut(\mc{G}), L \mapsto f(L) \otimes \1_{\mc{G}}
	\end{equation*}
	is a group homomorphism.
	The tensor product of 1-morphisms respects composition, i.e.
	\begin{equation*}
		(\mc{B}_{1} \otimes \mc{B}_{2}) \circ (\mc{A}_{1} \otimes \mc{A}_{2}) = (\mc{B}_{1} \circ \mc{A}_{1}) \otimes (\mc{B}_{2} \circ \mc{A}_{2}).
	\end{equation*}
	(See \cite[p.~59]{WaldorfPHD}.)
	By the Eckmann-Hilton argument, the operations $\otimes$ and $\circ$ agree on $\Aut(\mc{I}_{0})$, we thus have
	\begin{equation*}
		f(L) \circ f(L') = f(L) \otimes f(L').
	\end{equation*}
	Using \cite[Proposition 2.5.1]{WaldorfPHD}, we then compute
	\begin{align*}
		(f(L) \otimes \1_{\mc{G}}) \circ (f(L') \otimes \1_{\mc{G}}) &= (f(L) \circ f(L')) \otimes (\1_{\mc{G}} \circ \1_{\mc{G}}) \\
									     &= (f(L) \otimes f(L')) \otimes \1_{\mc{G}} \\
									     &= f(L \otimes L') \otimes \1_{\mc{G}}. \qedhere
	\end{align*}
\end{proof}
As a result of Proposition \ref{prop:TheKernel} we obtain an action of $G$ on $H^{1}(M,\Ab)$.
We shall see momentarily (Corollary \ref{cor:ConjugationOnH1}) that this action is, unsurprisingly, given by $([\beta], g) \mapsto [\action_{g}^{*} \beta]$.

\subsection{Functoriality of the group extension}\label{sec:functoriality}
The above construction is functorial in the following sense. 
Consider the category where the objects are group actions $G \curvearrowright (M, \cG_{M})$ 
that preserve the isomorphism class of the gerbe $\cG_{M}$ on $M$.
A \emph{morphism} from $G_1 \curvearrowright (M_1, \cG_{1})$ 
to ${G_2 \curvearrowright (M_2, \cG_{2})}$ is a triple $\Pi_{21} = (\pi_{21}, f_{12}, \Lambda_{21})$ consisting of a Lie group homomorphism $\pi_{21} \colon G_1 \rightarrow G_2$, a $G_1$-equivariant smooth map $f_{12} \colon M_2 \rightarrow M_1$, and an isomorphism 
$\Lambda_{12} \colon \cG_{2} \rightarrow f_{12}^*\cG_{1}$ of bundle gerbes.
Since $f_{12}$ is $G_1$-equivariant if $f_{12} \circ \action_{\pi(g)} = \action_{g} \circ f_{12}$ for all $g\in G_1$, the compatibility conditions 
on the triple $(\pi_{21}, f_{12}, \Lambda_{12})$ are summarized in Figure~\ref{fig:morphism}.
\begin{figure}[h!]
\centering
		\begin{tikzcd}[row sep=tiny]
			&G_1 \times M_1  \ar[r, "\action_1"] & M_1 & f_{12}^*\cG_{1} \ar[ddl, dashed]\\
			G_1 \times M_2 \ar[ur, "\mathrm{id}_{G_1}\times f_{12}"] \ar[dr, "\pi_{21} \times \mathrm{id}_{M_2}"']&& &\\
			&G_2 \times M_2 \ar[r, "\action_2"] & M_2 \ar[uu, "f_{12}"'] & \cG_{2} \ar[uu, "\rotatebox{90}{\(\sim\)}","\Lambda_{12}"'] \ar[l, dashed].
		\end{tikzcd}
\caption{A morphism $\Pi = (\pi_{21}, f_{12}, \Lambda_{12})$ from $G_1 \curvearrowright (M_1, \cG_{1})$ 
to ${G_2 \curvearrowright (M_2, \cG_{2})}$.}\label{fig:morphism}
\end{figure}

The composition of a morphism
$\Pi_{21} := (\pi_{21}, f_{12}, \Lambda_{12})$ from $G_1 \curvearrowright (M_1, \cG_{1})$ to ${G_2 \curvearrowright (M_2, \cG_{2})}$
with a morphism  
$\Pi_{32} := (\pi_{32}, f_{23}, \Lambda_{23})$
from $G_2 \curvearrowright (M_2, \cG_{2})$ to ${G_3 \curvearrowright (M_3, \cG_{3})}$
is defined by 
\[
\Pi_{32} \circ \Pi_{21} := (\pi_{32}\circ \pi_{21}, f_{12}\circ f_{23}, f_{23}^*\Lambda_{12} \circ \Lambda_{23}).
\]
That this is indeed a morphism from $(G_{1},M_{1},\mc{G}_{1})$ to $(G_{3},M_{3},\mc{G}_{3})$ can be verified by pasting together the appropriate commutative diagrams, as in Figure \ref{Fig:CompositionOfMorphisms}.
\begin{figure}[h!]
\centering
		\begin{tikzcd}[row sep=tiny]
			&&G_1 \times M_1  \ar[r, "\action_{1}"] & M_1 & f_{12}^*\cG_{1} \ar[ddl, dashed]& &f_{32}^*f_{21}^*\cG_{1}\ar[ddddlll, dashed]\\
			&G_1 \times M_2 \ar[ur, "\mathrm{id} \times f_{12}"] \ar[dr, "\pi_{21} \times \mathrm{id}"]&& &&&\\
		G_1 \times M_3\ar[ur, "\mathrm{id} \times f_{23}"] \ar[dr, "\pi_{21} \times \mathrm{id}"']& &G_2 \times M_2 \ar[r, "\action_{2}"] & 
		M_2 \ar[uu, "f_{12}"] & \cG_{2} \ar[uu, "\rotatebox{90}{\(\sim\)}","\Lambda_{12}"'] \ar[l, dashed]& 
		&f_{23}^*\cG_{2}\ar[uu,  "\rotatebox{90}{\(\sim\)}","f_{32}^*\Lambda_{12}"'] \ar[ddlll, dashed]\\
			&G_2 \times M_3 \ar[ur, "\mathrm{id} \times f_{23}"'] \ar[dr, "\pi_{32} \times \mathrm{id}"']&& &&&\\
			&&G_3 \times M_3  \ar[r, "\action_{3}"] & M_3 \ar[uu, "f_{32}"]&& & \cG_{3} \ar[lll, dashed] \ar[uu, "\rotatebox{90}{\(\sim\)}","\Lambda_{23}"'].
		\end{tikzcd}
\caption{The composition of $(\pi_{21}, f_{12}, \Lambda_{12})$ with  $(\pi_{32}, f_{23}, \Lambda_{23})$.}\label{Fig:CompositionOfMorphisms}
\end{figure}
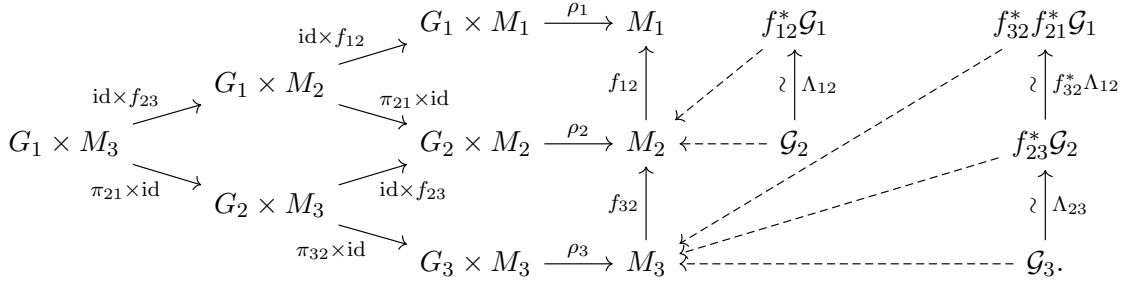

One easily checks that composition is associative;
\begin{eqnarray*}
\Pi_{43} \circ (\Pi_{32}\circ \Pi_{21}) & = & (\pi_{43}\circ\pi_{32}\circ \pi_{21}, f_{12}\circ f_{23} \circ f_{34}, (f_{23}\circ f_{34})^*\Lambda_{12} \circ f_{34}^*\Lambda_{23} \circ \Lambda_{34})\\
& = & (\Pi_{43}\circ \Pi_{32})\circ \Pi_{21}.
\end{eqnarray*}
Recall that to every object $G \curvearrowright (M, \cG_{M})$ we can assign the abelian extension $\widehat{G}_{\cG_{M}}$.
A morphism $\Pi = (\pi, f, \Lambda)$ from $G \curvearrowright (M, \cG_{M})$ 
to ${H \curvearrowright (N, \cG_{N})}$
induces a map $\widehat{\Pi} \colon \widehat{G}_{\cG_{M}} \rightarrow \widehat{H}_{\cG_{N}}$ by
\begin{equation}\label{eq:functorMorphism}
	\widehat{\Pi}(g, A) := (\pi(g), \Lambda^{-1} \circ f^* A \circ \action_{\pi(g)}^* \Lambda).
\end{equation}

\begin{figure}[h!]
\begin{equation*}
	\rho^{*}_{\pi(g)}\mc{G}_{N} \xrightarrow{\rho^{*}_{\pi(g)}\Lambda} \rho^{*}_{\pi(g)}f^{*}\mc{G}_{M} = f^{*}\rho^{*}_{g} \mc{G}_{M} \xrightarrow{f^{*}A} f^{*}\mc{G}_{M} \xrightarrow{\Lambda^{-1}} \mc{G}_{N}.
\end{equation*}
\caption{The morphism $\Lambda^{-1} \circ f^{*}A \circ \rho^{*}_{\pi(g)}\Lambda$ from $\rho_{\pi(g)}^{*}\mc{G}_{N}$ to $ \mc{G}_{N}$.}
\end{figure}

\begin{proposition}[Functoriality]\label{prop:DiscreteFunctoriality}
The map $\widehat{\Pi} \colon \widehat{G}_{\cG_{M}} \rightarrow \widehat{H}_{\cG_{N}}$ is a morphism
of abelian extensions, which reduces to 
$A \mapsto \Lambda^{-1} \circ f^*A \circ \Lambda$ on the kernel.
	\begin{equation}\label{functorialGpMor}
		\begin{tikzcd}
			H^{1}(M, \Ab) \ar[d, "\Lambda^{-1} \circ f^*(\,\cdot\,) \circ \Lambda"] \ar[r] & \widehat{H}_{\cG_{M}} \ar[r] \ar[d, "\widehat{\Pi}"] & G \ar[d, "\pi"] \\
			H^{1}(N, \Ab) \ar[r] & \widehat{G}_{\cG_{N}} \ar[r] & H.
		\end{tikzcd}
	\end{equation}
This assignment is functorial; if $\Pi_{31} = \Pi_{32}\circ \Pi_{21}$, then $\widehat{\Pi}_{31} = \widehat{\Pi}_{32}\circ \widehat{\Pi}_{21}$.
\end{proposition}
\begin{proof}
This is a straightforward calculation using the functoriality of the pullback. In detail, to see that $\widehat{\Pi}$ is a homomorphism, note that (using Equation \eqref{eq:Product}),
\begin{eqnarray*}
\widehat{\Pi}(g_1, A_1) \widehat{\Pi}(g_2, A_2)&=& \left(\pi(g_1)\pi(g_2), \Big(\Lambda^{-1} \circ f^*A_2 \circ \action_{\pi(g_2)}^*\Lambda\Big) 
\circ
\action_{\pi(g_2)}^*\Big(\Lambda^{-1} \circ f^*A_1 \circ \action_{\pi(g_1)}^*\Lambda\Big) \right)\\
&=& \left(
\pi(g_1g_2),  \Lambda^{-1} \circ f^*A_2 \circ (f\circ \action_{\pi(g_2)})^*A_1 \circ (\action_{\pi(g_1)}\circ \action_{\pi(g_2)})^*\Lambda
\right)\\
&=&\left(
\pi(g_1g_2), \Lambda^{-1} \circ f^*(A_2 \circ \action_{g_2}^* A_1)\circ \action_{\pi(g_1g_2)}^*\Lambda
\right) = \widehat{\Pi}((g_1,A_1)(g_2,A_2)),
\end{eqnarray*}
where the third equality uses $f\circ \action_{\pi(g_2)} = \action_{g_2}\circ f$.

Now we come to functoriality, let
$\Pi_{21} = (\pi_{21}, f_{12}, \Lambda_{12})$ be a morphism from $(G_1, M_1, \cG_{1})$ to $(G_2, M_2, \cG_{2})$ and let
$\Pi_{32} = (\pi_{32}, f_{23}, \Lambda_{23})$ be a morphism $(G_2, M_2, \cG_{2})$ to $(G_3, M_3, \cG_{3})$ and set $\Pi_{31} = \Pi_{32} \circ \Pi_{21} = (\pi_{32}\circ \pi_{21}, f_{12}\circ f_{23}, f_{23}^*\Lambda_{12} \circ \Lambda_{23})$.
We then have
\begin{align*}
	\widehat{\Pi}_{32}(\widehat{\Pi}_{21}(g,A))&= \widehat{\Pi}_{32}(\pi_{21}(g),\Lambda_{12}^{-1} \circ f^{*}_{12}A \circ \rho^{*}_{\pi_{21}(g)} \Lambda_{12}) \\
	&= (\pi_{32} \circ \pi_{21} (g), \Lambda_{23}^{-1} \circ f_{23}^{*} ( \Lambda_{12}^{-1} \circ f^{*}_{12}A \circ \rho^{*}_{\pi_{21}(g)} \Lambda_{12} ) \circ \rho^{*}_{\pi_{32} \circ \pi_{21}(g)} \Lambda_{23}) \\
	&= (\pi_{32} \circ \pi_{21} (g), \Lambda_{23}^{-1} \circ f_{23}^{*} \Lambda_{12}^{-1} \circ (f_{12} \circ f_{23})^{*}A \circ f_{23}^{*}(\rho^{*}_{\pi_{21}(g)} \Lambda_{12} ) \circ \rho^{*}_{\pi_{32} \circ \pi_{21}(g)} \Lambda_{23}) \\
	&= (\pi_{32} \circ \pi_{21} (g), (f_{23}^{*} \Lambda_{12} \circ \Lambda_{23})^{-1} \circ (f_{12} \circ f_{23})^{*}A \circ \rho^{*}_{\pi_{32} \circ \pi_{21}(g)} (f_{23}^{*} \Lambda_{12} \circ \Lambda_{23})) \\
	&= \widehat{\Pi}_{31}(g,A).
\end{align*}
	Here, we have used the condition $f_{23} \circ \rho_{\pi_{32}(h)} = \rho_{h} \circ f_{23}$, which implies $\rho^{*}_{\pi_{32}(h)} f_{23}^{*} = f_{23}^{*} \rho_{h}^{*}$, with $h = \pi_{21}(g)$, to go from the third line to the fourth.
\end{proof}

As a special case, suppose that the action of $H$ on $M$ preserves the isomorphism class of a bundle gerbe $\cG$.
Suppose that $\pi: G \rightarrow H$ is a group homomorphism.
We then set $\Pi = (\pi, \1,\1)$.
It follows that $\widehat{\Pi}(g,A) = (\pi(g),A)$.
In this context we are justified in writing $\widehat{\pi} := \widehat{\Pi}$, and Proposition \ref{functorialGpMor} specializes to the following corollary.

\begin{corollary}[Functoriality under group homomorphisms]
The map $\widehat{\pi}\colon \widehat{G}_{\cG} \rightarrow \widehat{H}_{\cG}$ is
a morphism of abelian extensions that reduces to the identity on $H^1(M,\Ab)$,
	\begin{equation}\label{functorialGp}
		\begin{tikzcd}
			H^{1}(M, \Ab) \ar[d, equal] \ar[r] & \widehat{G}_{\cG} \ar[r] \ar[d, "\widehat{\pi}"] & G \ar[d, "\pi"] \\
			H^{1}(M, \Ab) \ar[r] & \widehat{H}_{\cG} \ar[r] & H.
		\end{tikzcd}
	\end{equation}
This assignment is functorial; if 
$\pi_{31} = \pi_{32}\circ \pi_{21}$, then $\widehat{\pi_{31}} = \widehat{\pi_{32}}\circ \widehat{\pi_{21}}$.
\end{corollary}

\subsection{Interpretation in terms of the transgressed bundle}\label{sec:transgression}

Recall that the transgression of $\mc{G}$ is a principal $\Ab$-bundle $\mathscr{T}(\mc{G}) \rightarrow LM$, equipped with a connection $\nabla$, and a fusion structure $\mu$. In what follows, it will be convenient to realize $\widehat{G}_{\mc{G}}$ as a group of automorphisms of $\mathscr{T}(\mc{G}) \rightarrow LM$ that preserve
 both $\nabla$ and $\mu$. 


Indeed, let $\Aut_{G}(\mathscr{T}(\mc{G}), \nabla, \mu)$ be the group whose elements are pairs $(g, \Phi)$ of an element $g\in G$, together with a 
(not necessarily vertical) isomorphism $\Phi$ of 
$\mathscr{T}(\mc{G}) \rightarrow LM$ which preserves both $\nabla$ and $\mu$, and which covers the action $L\action_{g}$ of $g$ 
on the loop space $LM$. 

\begin{figure}[h!]
\centering
\begin{equation}
\begin{tikzcd}\label{diag:localAutomorphism}
		\mathscr{T}(\mc{G}) \ar[r, "\Phi"] \ar[d] & \mathscr{T}(\mc{G}) \ar[d] \\
		LM \ar[r,"L\action_{g}"] & LM.
	\end{tikzcd}
\end{equation}
\caption{Commutative diagram for $(g, \Phi) \in \Aut_{G}(\mathscr{T}(\mc{G}), \nabla, \mu)$.}
\end{figure}
We construct an isomorphism $\tgress_{L} \colon \widehat{G}_{\cG}\stackrel{\sim}{\rightarrow}\Aut_{G}(\mathscr{T}(\mc{G}), \nabla, \mu)$.
For 
$(g, A) \in \widehat{G}_{\cG}$,
transgression of $A$ yields a vertical isomorphism $\mathscr{T}(A) \colon \mathscr{T}(\rho_{g}^*\mc{G}) \rightarrow  \mathscr{T}(\mc{G})$ that preserves both $\nabla$ and $\mu$. Since $\tgress(\action_{g}^{*}\mc{G}) = (L\action_{g})^{*}\tgress(\mc{G})$ by Proposition~\ref{prop:transgressionAndPullback}, 
we obtain an element of $\Aut_{G}(\mathscr{T}(\mc{G}), \nabla, \mu)$ by concatenating 
$\mathscr{T}(A)^{-1}$ with the canonical map $(L\action_{g})_* \colon (L{\action_{g}})^*\mathscr{T}(\mc{G}) \rightarrow \mathscr{T}(\mc{G})$, as illustrated in the following diagram:
\begin{equation}\label{diag:ConstructTl}
	\begin{tikzcd}
		\mathscr{T}(\mc{G}) \arrow[r, "\mc{T}(A)^{-1}"] \arrow[d] & \mathscr{T}(\action_{g}^{*} \mc{G}) \arrow[r, equal] \arrow[d] & (L\action_{g})^{*}\mathscr{T}(\mc{G}) \arrow[r,"(L\action_{g})_*"] \arrow[d] & \mathscr{T}(\mc{G}) \arrow[d] \\
		LM \arrow[r, equal] & LM \arrow[r, equal] & LM \arrow[r, "L\action_{g}"] & LM.
	\end{tikzcd}
\end{equation}

For $\beta \in H^{1}(M, \Ab) \cong \Hom(H_{1}(M,\Z),\Ab)$, we denote by $\overline{\beta}$
the vertical automorphism of $\tgress(\mc{G})$ which is fibrewise given by $\overline{\beta}:\tgress(\mc{G})_{\gamma} \rightarrow \tgress(\mc{G})_{\gamma}, p \mapsto p \cdot \beta([\gamma])^{-1}$.


\begin{proposition}\label{prop:GroupTransgression}
	Transgression yields a group isomorphism 
	\begin{align*}
		\tgress_{L}:\widehat{G}_{\cG} &\rightarrow \Aut_{G}(\mathscr{T}(\mc{G}), \nabla, \mu), \\
		(g, A) & \mapsto (g, (L\action_{g})_{*} \circ \mathscr{T}(A)^{-1})
	\end{align*}
	with the property that $\tgress_{L}(\beta) = (\1,\overline{\beta})$ for all $\beta \in H^{1}(M,\Ab)$. It is an isomorphism of abelian extensions in the sense that the following diagram commutes:
	\begin{equation}\label{diag:TransgressionOfDiffeo}
		\begin{tikzcd}
			H^{1}(M, \Ab) \ar[d, equal] \ar[r] & \widehat{G}_{\cG} \ar[r] \ar[d,"\tgress_{L}"] & G \ar[d, equal] \\
			H^{1}(M, \Ab) \ar[r] & \Aut_{G}(\mathscr{T}(\mc{G}), \nabla, \mu) \ar[r] & G.
		\end{tikzcd}
	\end{equation}
\end{proposition}
\begin{proof}
	Let $\gamma \in LM$ be arbitrary.
	To prove that $\tgress_{L}(\beta) = (\1,\overline{\beta})$ one may verify that both sides correspond to the fibrewise defined map $p \mapsto p \otimes \gamma^{*}\beta^{-1}$. Indeed, $\tgress(\beta):p \mapsto p \circ \gamma^{*} \beta^{-1}$, and $\beta(\gamma)$ is the holonomy of the pullback of $\beta$ along $\gamma$.	
	This tells us that the square on the left in Diagram \eqref{diag:TransgressionOfDiffeo} commutes. 
	That the square on the right commutes follows from Diagram \eqref{diag:ConstructTl}.
	
	Now, we prove that $\tgress_{L}$ is a group homomorphism.
	To that end, let $(g, A)$ and $(h, B)$ be elements of $\widehat{G}_{\mc{G}}$.
	We must then show that
	\begin{equation*}
		(L\action_{g})_{*} \circ \tgress(A)^{-1} \circ (L\action_{h})_{*} \circ \tgress(B)^{-1} \overset{!}{=} (L(\action_{gh}))_{*} \circ \tgress(B \circ \rho_{h}^{*} A)^{-1}.
	\end{equation*}
	For brevity, let us omit the symbol $\circ$.
	Making repeated use of Proposition~\ref{prop:transgressionAndPullback}, we then compute
	\begin{align*}
		(L\action_{g})_{*} \tgress(A)^{-1} (L\action_{h})_{*} \tgress(B)^{-1} &= (L\action_{g})_{*} (L\action_{h})_{*} (L\action_{h})_{*}^{-1} \tgress(A)^{-1} (L\action_{h})_{*} \tgress(B)^{-1} \\
								       &= (L(\action_{gh}))_{*} \tgress(\action_{h}^{*}A)^{-1} \tgress(B)^{-1} \\
								       &= (L(\action_{gh}))_{*} \tgress(B \action_{h}^{*} A)^{-1},
	\end{align*}
	as desired.
	Bijectivity of $\tgress_{L}$ follows immediately from the fact that $\tgress$ is an equivalence:
	To see that $\tgress_{L}$ is injective, note that $\tgress_{L}(g,A) = \tgress_{L}(h,B)$ implies $g = h$ and 
	$\tgress(A) = \tgress(B)$, so $A = B$ because $\tgress$ is faithful.
	For surjectivity, let $(g, \Phi) \in \Aut_{G}(\tgress(\mc{G}), \nabla, \mu)$.
	Then $(L_{\action_{g}})^{-1}_{*}\circ \Phi$ is a bundle automorphism 
	from $\tgress(\cG)$ to $(L_{\action_{g}})^*\tgress(\cG)$ that covers the identity map on $LM$, and as such it is
	a morphism in $\fusb^{\nabla_{\text{sf}}}(LM)$.
	Since $\tgress$ is a full functor, there exists a morphism $A$ in  $h_{1} \grb(M)$ such that $\tgress(A)^{-1} = (L_{\action_{g}})_{*}^{-1}\circ \Phi$.
	It then follows that $\tgress_{L}(g, A) = (g,\Phi)$.
%
\end{proof}
In particular, this allows us to see that the action by conjugation of $G$ on $H^{1}(M,\Ab) \subseteq \widehat{G}_{\cG}$ is simply pullback along the action on $M$.
\begin{corollary}\label{cor:ConjugationOnH1}
Conjugation of $\beta \in H^{1}(M,\Ab)$ by $\widehat{g} = (g, \Phi) \in \widehat{G}_{\cG}$ is given by
	\begin{equation*}
		\widehat{g}{\,}^{-1} \beta \,\widehat{g} = \action_{g}^{*}\beta.
	\end{equation*}
\end{corollary}
\begin{proof}
	Let $\gamma \in LM$ and $p \in \tgress(\mc{G})_{\gamma}$.
	Then 
		\begin{equation*}
		\widehat{g}{\,}^{-1} \beta \,\widehat{g}  (p) = \Phi^{-1} \beta \Phi(p) = \Phi^{-1} (\Phi(p) \cdot \beta(\rho_{g} \gamma)) = p \cdot \beta(\rho_{g} \gamma) = 
		\rho_{g}^{*}\beta(p).  \qedhere
	\end{equation*}
\end{proof}

\section{Lie group structures}\label{sec:smoothstructure}

Let $M$ be an orientable, connected, finite-dimensional manifold for which $H_1(M,\Z)$ is finitely generated.
Let $G$ be a Fr\'echet--Lie group with a smooth action $\rho \colon G \times M \rightarrow M$ that preserves the isomorphism class of a bundle gerbe $\cG$ on $M$.
In this section, we equip $\widehat{G}_{\mc{G}}$ with a (Fr\'echet) Lie group structure 
that turns the short exact sequence
\begin{equation}\label{eq:AbelianGroupExtensionIntro4}
	H^1(M, \Ab) \longrightarrow \widehat{G}_{\mc{G}} \longrightarrow G 
\end{equation}
into an abelian extension of (Fr\'echet) Lie groups.

\subsection{The central Lie algebra extension}

Before going to the group level, we first describe the corresponding Lie algebra extension
\begin{equation}\label{eq:AlgExtension2}
H^1(M, \ab) \rightarrow \widehat{\fg}_{\omega} \rightarrow \fg.
\end{equation}
This is a central extension, and, correspondingly,  the abelian extension \eqref{eq:AbelianGroupExtensionIntro4} will be central when restricted to the 
identity component of $G$.
If $\omega \in \Omega^3(M,\fz)$ is the curvature of $\cG$, then $\rho_{g}^*\omega = \omega$ for all $g\in G$. It follows that 
the Lie derivative $L_{X_{M}}\omega$ along the fundamental vector field $X_{M}$
of $X\in \fg$ is zero. Since $d\omega = 0$, we conclude that $i_{X_{M}}\omega$ is closed. 
\begin{proposition}
If the action of $G$ on $M$ preserves the isomorphism class of the bundle gerbe $\cG$, then 
the action is exact, in the sense that $i_{X_{M}}\omega$ is exact for all $X\in \fg$.
\end{proposition}
\begin{proof}
It suffices to show that the class $[i_{X_M}\omega]\in H^2_{\mathrm{dR}}(M)$ pairs to zero with a basis 
$\tau_1, \ldots, \tau_{N}$ of  $H_2(M,\R)$. By a result of Thom, we may assume that $\tau_i = \sigma_*[\Sigma]$ for
a closed, oriented 2-manifold with a smooth map $\sigma \colon \Sigma \rightarrow M$.
(cf.~\cite[Corollary 2.30]{Thom1954}
for the compact and \cite[\S 9 and \S 15]{ConnerFloyd1964} for the noncompact case).

Let 
$t \mapsto g(t)$ be a smooth curve in $G$ that passes through the identity at $t=0$.
Define $\sigma_t \colon \Sigma \rightarrow M$ by $\sigma_t := \rho_{g(t)} \circ \sigma$, and define the flow $\Phi \colon \Sigma \times \R \rightarrow M$
by $\Phi(s,t) := \sigma_t(s)$.
By Lemma~\ref{Lemma:HolonomyCurvature}, we then have 
\[
\mathrm{hol}_{\mc{G}}(\Sigma, \sigma_t) - \mathrm{hol}_{\mc{G}}(\Sigma, \sigma_0) = \int_{\Sigma \times [0,t]} \Phi^*\omega.
\]
Since $\rho_{g(t)}$ preserves the isomorphism class of $\cG$, it preserves the holonomy as well. It follows that 
\[  {\textstyle \frac{d}{dt}}\mathrm{hol}_{\mc{G}}(\Sigma, \sigma_t) = \int_{\Sigma} \sigma^*(i_{X_M}\omega) = 0,\]
so $[i_{X_M}\omega]$ pairs to zero with $\sigma_*[\Sigma]$.
%
\end{proof}

This allows us to construct the central extension \eqref{eq:AlgExtension2} along the lines of \cite[Proposition~5.2]{DJNV21}, 
cf.~also \cite{Ro95, Ne05}. 
Set $\ol{\Omega}{ }^{1}(M,\ab) := \Omega^{1}(M,\ab)/d\Omega^0(M,\ab)$,
and denote the class of $\psi \in \Omega^{1}(M,\lie{z})$ by $\overline{\psi} \in \ol{\Omega}{ }^{1}(M,\ab)$.
Since a 1-form is exact if and only if it integrates to zero against every $1$-cycle, the subspace $d\Omega^0(M) \subseteq \Omega^1(M)$
is closed, and $\ol{\Omega}{ }^{1}(M,\ab)$ is a Fr\'echet space with the quotient topology.

Since $i_{X_{M}}\omega$ is exact, we can thus define the Fr\'echet space
\begin{equation}\label{eq:LAcentralextension}
	\widehat{\fg}_{\omega} := \{(X,\ol{\psi}_{X}) \in \fg \times \ol{\Omega}{}^1(M,\ab)\,;\, i_{X_{M}} \omega = d\psi_{X}\}.
\end{equation}
The Lie bracket on $\widehat{\lie{g}}_{\omega}$ is given by
\begin{equation}\label{eq:definitionbracket}
	[(X,\ol{\psi}_{X}), (Y, \ol{\psi}_{Y})] = ([X,Y], \ol{i_{X_{M}} i_{Y_{M}}\omega}).
\end{equation}
It is skew-symmetric by definition, and the Jacobi identity 
follows from a straightforward computation \cite[Proposition 5.2]{DJNV21}.
The map $\widehat{\fg}_{\omega} \rightarrow \fg$ is the projection onto the first factor, 
and its kernel is of course precisely $H^1_{\dR}(M,\ab)$.

\subsection{The Fr\'echet--Lie group extension}

We now equip $\widehat{G}_{\mc{G}}$ with the structure of Fr\'{e}chet--Lie group in two steps:
\begin{enumerate}\itemsep-.25em
	\item We find a normal subgroup $\widehat{N} \subseteq \widehat{G}_{\mc{G}}$ and equip that group with a Lie group structure.
	\item We show that the conjugation action of $\widehat{G}_{\mc{G}}$ on $\widehat{N}$ is smooth, which allows us to transfer the charts of $\widehat{N}$ to the entire group $\widehat{G}_{\mc{G}}$.
\end{enumerate}

Let $N \subseteq G$ be the open subgroup that acts trivially on $\pi_0(LM)$.
Since this group is normal, the restriction of $\widehat{G}_{\mc{G}}$ to $N$, denoted by $\widehat{N}$, is normal as well.
Our goal, now, is to equip $\widehat{N}$ with a Lie group structure.

Before proceeding with this, we introduce some notation.
If $\gamma \in LM$, we write $[\gamma] \in \pi_{1}(M)$, $\langle \gamma \rangle \in \pi_{0}(LM)$ and $[[\gamma]] \in H_{1}(M,\Z)$ for the corresponding elements.
Since the connected components of $LM$ correspond to conjugacy classes in $\pi_1(M)$, the map
$\pi_1(M) \rightarrow \pi_0(LM)$ is a surjective class function.  In particular, the quotient map $\pi_{1}(M) \rightarrow \pi_{1}(M)/[\pi_{1}(M),\pi_{1}(M)] = H_{1}(M,\Z)$ factors through $\pi_{0}(LM)$.  If $\langle \gamma \rangle \in \pi_{0}(LM)$, we can therefore unambiguously write $[[\gamma]] \in H_{1}(M,\Z)$.

Let $\langle \gamma \rangle \in \pi_{0}(LM)$, and view $\langle \gamma \rangle$ as a subset of $LM$.
By definition, the action of $N$ on $LM$ restricts to an action on $\comp{\gamma}$.
We define $N^{\sharp}_{\comp{\gamma}}$ to be the group consisting of 
pairs $(n,\alpha)$ with $n\in N$, and $\alpha$ an automorphism of $\tgress(\mc{G})|_{\comp{\gamma}} \rightarrow \comp{\gamma}$ 
that preserves $\nabla$, and covers the action $L\rho_{n}$ of $n$ on $\comp{\gamma}$,
\begin{equation}
	\begin{tikzcd}
		\mathscr{T}(\mc{G})|_{\comp{\gamma}} \ar[r, "\alpha"] \ar[d] & \mathscr{T}(\mc{G})|_{\comp{\gamma}} \ar[d] \\
		\comp{\gamma} \ar[r,"L\action_{n}"] & \comp{\gamma}.
	\end{tikzcd}
\end{equation}
The following proposition is a slight variation on \cite[\S 5]{DJNV21}, where the Kostant-Kirillov-Souriau extensions for certain 
prequantum line bundles over nonlinear Grassmannians were constructed.

\begin{proposition}\label{Prop:vasteklasse}
	Let $G$ be a locally convex Lie group with an action on a connected manifold $M$ that preserves the isomorphism class of 
	a bundle gerbe $\cG$ on $M$.
	Then, for every connected component of $LM$, we have a unique central extension of locally convex Lie groups
	\begin{equation}\label{eq:NiLiegroep}
		\Ab \rightarrow N_{\comp{\gamma}}^{\sharp} \rightarrow N
	\end{equation}
	such that the action of $N^{\sharp}_{\comp{\gamma}}$ on $\tgress(\mc{G})|_{\comp{\gamma}}$ smoothly covers the action of $N_{\gamma}$
	on $\comp{\gamma}$.
	The corresponding Lie algebra cocycle $\psi \colon \fn \times \fn \rightarrow \ab$ is given by
	\begin{equation}\label{eq:Nicocycle}
		\psi(X,Y) = \oint_{\gamma} i_{X_M}i_{Y_M}\omega,
	\end{equation}
	where $X_M$ and $Y_M$ are the fundamental vector fields of $X,Y \in \fn$ on $M$, and $\omega \in \Omega^3(M, \ab)$ is the curvature of $\cG$.
\end{proposition}
The proof uses the following minor variation of \cite{NV03}.
Let $\Ab$ be a finite-dimensional abelian Lie group, let 
$\cP \rightarrow \cM$ be a principal $\Ab$-bundle over a connected locally convex manifold $\cM$, let $\nabla$ be a principal 
connection with curvature $\Omega \in \Omega^2(\cM, \ab)$,
and let $G$ be a locally convex Lie group with a smooth action on $\cM$ that preserves the holonomy.

\begin{theorem}\label{Thm:NeebVizman2003adaptation}
The group $G^{\sharp}$ of automorphisms of $(\cP, \nabla)$ that cover the action of $G$ on $\cM$ carries a locally convex Lie group
structure for which the action on $\cP$ is smooth.
\end{theorem}
\begin{proof}
We reduce this to the case where $\Ab$ is connected, which is precisely \cite[Theorem A.1]{DJNV21}. 
The quotient $\Sigma := \cP/\Ab_0$ of $\cP$ by the connected identity component  $\Ab_0 \subseteq \Ab$ is a discrete principal 
$\pi_0(\Ab)$-bundle over $\cM$, so the group $G'$ of bundle automorphisms of $\Sigma \rightarrow \cM$ that cover the action of $G$ on $\cM$
is a locally convex Lie group. 
Since $\cP \rightarrow \Sigma$ is a principal fibre bundle with connected structure group $\Ab_0$, it follows from \cite[Theorem A.1]{DJNV21} that the group $(G')^{\sharp}$ 
of connection-preserving automorphisms of $\cP \rightarrow \Sigma$ that cover the $G'$-action on $\Sigma$ is a locally convex Lie group 
with a smooth action on $\cP$. The group $G^{\sharp}$ of automorphisms of $\cP \rightarrow \cM$ that cover the $G$-action on $\cM$ 
is a subgroup of $(G')^{\sharp}$, and the composition of the inclusion $\iota \colon G^{\sharp} \hookrightarrow (G')^{\sharp}$ with the 
projection $\pi \colon (G')^{\sharp} \twoheadrightarrow G'$ is a homomorphism $h \colon G^{\sharp} \rightarrow G'$ that covers the surjection $G^{\sharp} \twoheadrightarrow G$. So $h(G^{\sharp}) \subseteq G'$ is an open subgroup, and $\pi^{-1}(h(G^{\sharp})) = \iota(G^{\sharp})$ is an open subgroup of 
$(G')^{\sharp}$.
\end{proof}

\begin{proof}[Proof of Proposition~\ref{Prop:vasteklasse}]
Since $\comp{\gamma} \subseteq LM$ is a connected locally convex manifold, and since the action of the (not necessarily connected) Lie group $N$ 
on $\comp{\gamma}$ preserves the holonomy of $\mathscr{T}(\mc{G})|_{\comp{\gamma}}$, it follows from
Theorem~\ref{Thm:NeebVizman2003adaptation} that 
\eqref{eq:NiLiegroep} is a central extension of locally convex Lie groups.
The manifold structure is obtained from the pullback of $\mathscr{T}(\mc{G})|_{\comp{\gamma}} \rightarrow \comp{\gamma}$ along an orbit map
$N \rightarrow \comp{\gamma} \colon n \mapsto \action_{n}\circ \gamma$, and does not depend on the choice of representative $\gamma$ by 
\cite[Lemma 3.2]{NV03}. By \cite[Remark 3.5]{NV03}, the corresponding Lie algebra cocycle is given by evaluating the curvature 
$\Omega$ of the transgressed line bundle in the fundamental vector fields on $\comp{\gamma}$ at a fixed loop $\gamma$ 
(this is the usual Kostant-Kirillov-Souriau extension \cite{Ko70}).
If $\cG$ has curvature $\omega\in \Omega^3(M, \ab)$, then by \cite[Proposition 4.3.2]{Wa16} the transgressed line bundle on $LM$ has curvature 
\[
\Omega_{\gamma}(u,v) = \oint_{S^1}i_{u}i_{v}\omega,
\]
where $v,w$ are tangent vectors in $T_{\gamma}LM = \Gamma(\gamma^*TM)$. Plugging the fundamental 
vector fields of $N$ on $LM$ into this $2$-form yields \eqref{eq:Nicocycle}.
\end{proof}
\begin{remark}
We needed $M$ to be a manifold because the construction of the transgressed line bundle in \cite{Wa16} uses diffeological spaces, 
and porting these results to the category of manifolds uses the fact that the category of smooth manifolds embed fully faithfully 
into the category of diffeological spaces. The Lie group $G$ need only be locally convex
at this point, but we will require it to be Fr\'echet later on.
\end{remark}

Let $\langle \gamma \rangle \in \pi_{0}(LM)$ be an arbitrary connected component of $LM$.
Since $N$ preserves $\langle \gamma \rangle$, 
every $\widehat{n} \in \widehat{N}$ gives rise to an automorphism
of $\mathscr{T}(\mc{G})|_{\comp{\gamma}} \rightarrow \comp{\gamma}$ that preserves the connection $\nabla$.
We therefore obtain a homomorphism $\widehat{N} \rightarrow N^{\sharp}_{\comp{\gamma}}$, and 
we denote its image by $\widehat{N}_{\comp{\gamma}}$.
The images of $\widehat{n}$ in groups $N^{\sharp}_{\comp{\gamma}}$ for various $\langle \gamma \rangle$ are related by the fact that $\widehat{n}$ preserves the fusion product.
We will use this to show that for every component $\comp{\gamma}$, the group $\widehat{N}_{\comp{\gamma}}$ is an embedded Lie subgroup of $N_{\comp{\gamma}}^{\sharp}$.

Since any $k\in K$ covers the identity $\one \in G$, it acts by a vertical automorphism
$\alpha_k$ on $(\mathscr{T}(\mc{G}),\nabla)$. Since $\alpha_{k}$ preserves the connection, it is given by a locally constant function $\Lambda$ on $LM$ with values 
in $\Ab$.
Since the connected components of $LM$ correspond to conjugacy classes in $\pi_1(M)$, we can consider $\Lambda$ as 
a class function $\Lambda \colon \pi_1(M) \rightarrow \Ab$.
If $p_{\gamma} \in \mathscr{T}(\mc{G})$ is an element of the $\Ab$-bundle over $\gamma \in LM$, we thus have
\begin{equation}
	\alpha_k (p_{\gamma}) = p_{\gamma} \cdot \Lambda([\gamma]).
\end{equation}
Using Proposition \ref{prop:GroupTransgression}, we identify the group of vertical automorphisms of $\mathscr{T}(\mc{G}) \rightarrow LM$ 
that preserve the fusion product with $K$, (in particular $K \subseteq \widehat{N}$); moreover we identify the group of vertical automorphisms of 
the transgressed line bundle $\mathscr{T}(\mc{G})|_{\comp{\gamma}} \rightarrow \comp{\gamma}$ with $\Ab$.
\begin{lemma}\label{lem:verticaal}
	 Under the isomorphism of $K$ with $H^1(M, \Ab)\simeq \Hom(H_1(M,\Z), \Ab)$, 
	 the restriction of the map $\widehat{N} \rightarrow N^{\sharp}_{\comp{\gamma}}$ to 
	 $K\subseteq \widehat{N}$ corresponds to evaluation 
	\[\ev_{[[\gamma]]} \colon \Hom(H_1(M,\Z), \Ab) \rightarrow \Ab\]
	at the homology class $[[\gamma]] \in H_1(M,\Z)$ defined by the loop $\gamma$.
\end{lemma}
\begin{proof} 
	A vertical automorphism $\alpha$ of $(\mathscr{T}(\mc{G}), \nabla)$ corresponds to a 1-cochain $\Lambda \colon \pi_1(M) \rightarrow \Ab$
	that is constant on conjugacy classes, $\alpha(p_{\gamma}) = p_{\gamma} \Lambda([\gamma])$.
	It preserves the fusion structure $\mu$ if and only if
	\[\mu(p_{12}\otimes p_{23}) = \Lambda([\ol{\gamma}_{3}*\gamma_1])^{-1}\mu(p_{12}\Lambda([\ol{\gamma}_{2}*\gamma_1])\otimes p_{23}\Lambda([\ol{\gamma}_{3}*\gamma_2]))\]
	for all $p_{12}\otimes p_{23}$ over $(\gamma_1,\gamma_2,\gamma_3) \in PM^{[3]}$.
	This is the case if and only if
	\be\label{eq:gesloten}
	\delta \Lambda ([\gamma_{23}],[\gamma_{13}],[\gamma_{12}]) := \Lambda([\gamma_{13}])^{-1}\Lambda([\gamma_{12}])\Lambda([\gamma_{23}]) = 1
	\ee
	for all $[\gamma_{ij}] = [\ol{\gamma}_{j}*\gamma_i] \in \pi_1(M)$, that is, if $\Lambda$ is a closed $1$-cochain.
	This is the same as a group homomorphism $\Lambda \colon \pi_1(M) \rightarrow \Ab$.
	Since a group homomorphism from $\pi_1(M)$ into the abelian group $\Ab$ is automatically constant on conjugacy classes,
	we obtain all group homomorphisms in this way.
	It follows that the group of vertical automorphisms of $\mu$ is
	\begin{eqnarray*}
		H^1(\pi_1(M),\Ab) &=& \Hom(\pi_1(M),\Ab) = \Hom\Big(\pi_1(M)/[\pi_1(M),\pi_1(M)],\Ab\Big)\\
				    &=& \Hom(H_1(M,\Z),\Ab) = H^1(M,\Ab),
	\end{eqnarray*}
	as required. 
	Since the commutator subgroup $[\pi_1(M),\pi_1(M)]$ is contained in the kernel of $\Lambda$, evaluating $\Lambda$ in the conjugacy class of 
	$[\gamma] \in \pi_1(M)$ is the same as evaluating it in the homology class $[[\gamma]] \in H_1(M,\Z) \simeq \pi_1(M)/[\pi_1(M),\pi_1(M)]$.
\end{proof}
\begin{remark}
In fact, this yields an alternative proof of Proposition~\ref{prop:TheKernel}.
\end{remark}

Suppose that $H_1(M,\Z)$ is finitely generated.
If $\Ab$ is connected and $[[\gamma]]$ is of infinite order, then 
the image $\Gamma_{\comp{\gamma}}$ of the evaluation map $\ev_{[[\gamma]]} \colon K \rightarrow \Ab$
is all of $\Ab$. 
In this case $\widehat{N}_{\comp{\gamma}} = N^{\sharp}_{\comp{\gamma}}$ is a Lie group by Proposition~\ref{Prop:vasteklasse}.
If, on the other hand, $[[\gamma]]\in H_1(M,\Z)$ is of finite order $k$,
then the image $\Gamma_{\comp{\gamma}}$ of the evaluation map is a discrete subgroup of $\Ab$.
We show that in this case too, the group $\widehat{N}_{\comp{\gamma}}$ has a natural Lie group structure 
that makes 
\begin{equation}\label{eq:weereenrijtje}
	\Gamma_{\comp{\gamma}} \rightarrow \widehat{N}_{\comp{\gamma}} \rightarrow N
\end{equation}
into a central extension of Lie groups.	

\begin{lemma}\label{lem:homomorphisms}
Denote by $\Ab_{k} \subseteq \Ab$ the subgroup of elements with $z^k = e$ if $k >0$ is finite, and set $\Ab_{\infty} := \Ab$ for $k = \infty$.
If $[[\gamma]]$ is of order $k \in \N^{>0} \cup \{\infty\}$, then $\Gamma_{\comp{\gamma}} \subseteq \Ab_{k}$ and 
$\Gamma_{\comp{\gamma}} \cap \Ab_0  = \Ab_k \cap \Ab_0$. In particular, $\Gamma_{\comp{\gamma}} = \Ab_k$ if $\Ab$
is connected.
\end{lemma}
This can be proven by elementary means from the classification of finitely generated abelian groups and the fact 
that the connected identity component $\Ab_0$ of $\Ab$ is isomorphic to $\R^n \times \U(1)^{m}$.

\begin{proposition}\label{prop:CoreTechnicalResult}
For every connected component $\comp{\gamma} \subseteq LM$, the group $\widehat{N}_{\comp{\gamma}}$ is an embedded Lie subgroup of $N^{\sharp}_{\comp{\gamma}}$, 
and \eqref{eq:weereenrijtje} is a central extension of Fr\'echet--Lie groups.
\end{proposition}
\begin{proof}
The case that $[[\gamma]]$ is of infinite order in $H^{1}(M, \Z)$ is dispensed with easily: in this case $\Gamma_{\comp{\gamma}}$ is an open subgroup of $\Ab$ by Lemma~\ref{lem:homomorphisms}, so that $\widehat{N}_{\comp{\gamma}}$ is an open subgroup of $ N^{\sharp}_{\comp{\gamma}}$. 

If $[[\gamma]]$ is of finite order in $H^1(M, \Z)$, then 
$\Gamma_{\comp{\gamma}} \subseteq \Ab$ is a discrete subgroup.
The group $\widehat{N}_{\comp{\gamma}}\subseteq N^{\sharp}_{\comp{\gamma}}$ is an embedded Lie subgroup that makes $\widehat{N}_{\comp{\gamma}} \rightarrow N$ a discrete cover if and only if there exists a smooth local section 
$\sigma_{\comp{\gamma}} \colon U \rightarrow N^{\sharp}_{\comp{\gamma}}$ on an identity neighbourhood 
$U\subseteq N$
that takes values in $\widehat{N}_{\comp{\gamma}}$, or, equivalently, if  $\sigma_{\comp{\gamma}}(U)\Gamma_{\comp{\gamma}} \subseteq \widehat{N}_{\comp{\gamma}}$.
In fact, since $\widehat{N}_{\comp{\gamma}} \subseteq N^{\sharp}_{\comp{\gamma}}$ an embedded Lie subgroup if and only if there exists an open identity neighbourhood $W \subseteq N^{\sharp}_{\comp{\gamma}}$ such that $\widehat{N}_{\comp{\gamma}} \cap W \subseteq N^{\sharp}_{\comp{\gamma}}\cap W$ is an embedded submanifold,
it suffices to show that $\sigma_{\comp{\gamma}}(U)\Gamma_{\comp{\gamma}}\cap W = \widehat{N}_{\comp{\gamma}} \cap W$.
We will exhibit a section with this property.

Choose, for every connected component $\comp{\gamma}$ of $LM$, a symmetric neighbourhood of the identity, say $U_{\comp{\gamma}} \subseteq N$, together with a smooth local section 
\[
	\sigma_{\comp{\gamma}} \colon U_{\comp{\gamma}} \rightarrow N^{\sharp}_{\comp{\gamma}} 
\]
of the $\Ab$-bundle $N_{\comp{\gamma}}^{\sharp} \rightarrow N$, such that $\sigma_{\comp{\gamma}}(\one) = \one$.
If $n \in U_{\comp{\gamma}}$, and $\hat{N} \ni \hat{n} \mapsto n$, then $\widehat{n} \sigma_{\gamma}(n)^{-1}$ is a vertical, connection-preserving automorphism of 
the transgressed line bundle over $\comp{\gamma}$.
It follows that 
for every $p \in \mathscr{T}(\mc{G})|_{\comp{\gamma}}$, we have 
\begin{equation}\label{eq:Lambdasnede}
\widehat{n} p = \sigma_{\comp{\gamma}}(n) p \cdot \Lambda(\widehat{n}, [\gamma])
\end{equation}
for an element $\Lambda(\widehat{n}, [\gamma]) \in \Ab$ that depends only on $\widehat{n}$ and on the conjugacy class of $[\gamma]\in \pi_1(M)$.
Let $(\gamma_1, \gamma_2, \gamma_3) \in PM^{[3]}$, set $\gamma_{ij} = \ol{\gamma}_{j}*\gamma_i \in PM^{[2]}$, and let $p_{ij} \in \mathscr{T}(\mc{G})$ be elements of the principal $\Ab$-bundle $\mathscr{T}(\mc{G})$ over $\gamma_{ij}$.
Using the fact that $\hat{n}$ preserves the fusion product, i.e.~$\mu(\widehat{n}p_{12} \otimes \widehat{n} p_{23}) = \widehat{n}\mu(p_{12}\otimes p_{23})$, we find that
\begin{equation}\label{nHatFusionSection}
\mu(\sigma_{\gamma_{12}}(n)p_{12} \otimes  \sigma_{\gamma_{23}}(n)p_{23}) \cdot \Lambda(\widehat{n}, [\gamma_{12}])\Lambda(\widehat{n}, [\gamma_{23}])
=
\sigma_{\gamma_{13}}(n)\mu(p_{12} \otimes p_{23}) \cdot \Lambda(\widehat{n}, [\gamma_{13}]) 
\end{equation}
for $n \in U_{[\gamma_{12}]} \cap U_{[\gamma_{23}]}\cap U_{[\gamma_{13}]}$.
Note that $[\gamma_{12}]*[\gamma_{23}] = [\gamma_{13}]$ in $\pi_1(M)$; what's more, for every $x,y \in \pi_{1}(M)$, there exists an element $(\gamma_{1},\gamma_{2},\gamma_{3}) \in LM^{[3]}$ such that $[\gamma_{12}] = x$ and $[\gamma_{23}] = y$, for which we automatically have $x * y = [\gamma_{13}]$.
So if we denote the boundary of the 1-cochain $\Lambda(\widehat{n}, \,\cdot\,) \colon \pi_1(M) \rightarrow \Ab$ by
\[
\delta \Lambda(\widehat{n}, [\gamma_{12}], [\gamma_{23}] ):= \Lambda(\widehat{n}, [\gamma_{12}]) \Lambda(\widehat{n}, [\gamma_{23}])
\Lambda(\widehat{n}, [\gamma_{12}] * [\gamma_{23}])^{-1},
\]
then \eqref{nHatFusionSection} 
becomes 
\begin{equation}\label{eq:fusionboundary}
\delta \Lambda(\widehat{n}, [\gamma_{12}], [\gamma_{23}] )
=
\frac
{\sigma_{\gamma_{13}}(n)\mu(p_{12} \otimes p_{23})}
{\mu(\sigma_{\gamma_{12}}(n)p_{12} \otimes  \sigma_{\gamma_{23}}(n)p_{23})},
\end{equation}
where the r.h.s.~denotes the unique element of $\Ab$ that takes the denominator to the numerator (both live in 
the same fibre of the transgressed principal $\Ab$-bundle over $\gamma_{13}$).

Note that the expression in \eqref{eq:fusionboundary} depends on $\widehat{n} \in \widehat{N}$ only through $n \in N$ (because the r.h.s.\ does so), 
and it depends on $p_{ij} \in \mathscr{T}(\mc{G})$ only through the conjugacy class of $[\gamma_{ij}] \in \pi_1(M)$ (because the l.h.s.\ does so).
Further, since $\Lambda(\widehat{n},\,\cdot\,)$ is a class function, the coboundary 
\begin{equation}\label{DefB}
B(n,[\gamma_{12}], [\gamma_{23}]) := \delta \Lambda(\widehat{n}, [\gamma_{12}], [\gamma_{23}] )
\end{equation}
is conjugation invariant, and symmetric in $[\gamma_{12}]$ and $[\gamma_{23}]$.
It satisfies $B(\one, [\gamma_{12}], [\gamma_{23}]) = 1$ because $\Lambda(\one, [\gamma]) = 1$ for all $[\gamma]\in \pi_1(M)$.

Further, $B$
depends smoothly on $n$ for fixed $[\gamma_{12}], [\gamma_{23}] \in \pi_1(M)$. 
Indeed, since $n \mapsto \sigma_{\gamma_{13}}(n)$ is a smooth map from $U_{\gamma_{13}}$ to $N^{\sharp}_{\gamma_{13}}$,
and since $N^{\sharp}_{\gamma_{13}}$ acts smoothly on the diffeological bundle $l^*\mathscr{T}(\mc{G}) \rightarrow PM^{[2]}$, the map
\[N \supseteq U_{\gamma_{13}} \rightarrow l^*\mathscr{T}(\mc{G}), \quad n \mapsto \sigma_{\gamma_{13}}(n)\mu(p_{12} \otimes p_{23})\]
is a smooth map of diffeological spaces that covers the action of $N$ on the base $PM^{[2]}$.
In the same vein, 
\[
n \mapsto \mu(\sigma_{\gamma_{12}}(n)p_{12} \otimes  \sigma_{\gamma_{23}}(n)p_{23})
\]
is a smooth map from $U_{\gamma_{12}} \cap U_{\gamma_{23}}$ to $ l^*\mathscr{T}(\mc{G})$ that covers the same action on the base.
It follows that the quotient $n \mapsto B(n,[\gamma_{12}], [\gamma_{23}])$ is a smooth map $U_{\gamma_{12}} \cap U_{\gamma_{23}} \cap U_{\gamma_{13}}
\rightarrow \Ab$ of diffeological spaces.
Since the category of Fr\'echet manifolds embeds fully faithfully into the category of diffeological spaces \cite{Fr81, Lo92}, this map is smooth in 
the category of Fr\'echet manifolds as well.

We shall make repeated use of the following trick.
Suppose that for every connected component $\comp{\gamma}$,
we have a symmetric identity neighbourhood $V_{\comp{\gamma}} \subseteq N$,  and a
smooth function 
\[V_{\comp{\gamma}} \rightarrow \Ab, n \mapsto A(n, [\gamma])\]
with $A(\one, [\gamma]) = 1$.
Then $n \mapsto  \sigma_{\comp{\gamma}}(n) A(n, [\gamma])$ is a smooth local section of $N^{\sharp}_{\comp{\gamma}} \rightarrow N$ on the symmetric identity neighbourhood $U_{\comp{\gamma}}\cap V_{\comp{\gamma}}$.
Correspondingly, $\Lambda(\widehat{n}, [\gamma])$ changes to 
\[\widetilde{\Lambda}(\widehat{n}, [\gamma]) = \Lambda(\widehat{n}, [\gamma])A(n, [\gamma])^{-1},\] and 
$B(n, [\gamma_{12}], [\gamma_{23}])$ changes to 
\[\widetilde{B}(n, [\gamma_{12}], [\gamma_{23}]) = B(n, [\gamma_{12}], [\gamma_{23}]) \cdot \Big(\delta A(n, [\gamma_{12}], [\gamma_{23}])\Big)^{-1},\]
where $\delta A(n, [\gamma], [\tau]) := A(n, [\gamma])A(n, [\tau]) A^{-1}(n, [\gamma]*[\tau])$
as before. 
Since $A(n, [\gamma] )$ is conjugation invariant, we have $A(n, [\gamma]*[\tau]) = A(n, [\tau]*[\gamma])$ for all $[\gamma], [\tau] \in \pi_1(M)$. 
This implies that $B$ is conjugation invariant, and that it remains symmetric in the last two arguments.
Similarly, $\widetilde{\Lambda}(\one, [\gamma]) = 1$ and $\widetilde{B}(\one, [\gamma_{12}], [\gamma_{23}]) = 1$ because $A(\one, [\gamma]) = 1$
for all $\gamma$.

\begin{enumerate}
\item We may assume that $B(n,[\gamma],[*]) = B(n,[*],[\gamma]) = 1$ for all $[\gamma] \in \pi_{1}(M)$, and $\Lambda(\hat{n},[*]) = 1$.

 This is achieved by replacing $\sigma_{\comp{\gamma}}$ by $\widetilde{\sigma}_{\comp\gamma}:n \mapsto \sigma_{\comp\gamma}(n) A(n, [\gamma])$ where $A(n,[*]) := B(n, [*], [*]) = \Lambda(\hat{n},[*])$ and $A(n,[\gamma]) := 1$ for $\langle \gamma \rangle \neq \langle * \rangle$.
 Since $\delta A (n, [*], [*]) = A(n, [*])$, we obtain $\widetilde{B}(n, [*], [*]) = 1$.
 Moreover, we have $\widetilde{\Lambda}(\hat{n},[*]) = \Lambda(\hat{n},[*]) A(\hat{n},[*])^{-1} = \Lambda(\hat{n},[*]) B(n,[*],[*])^{-1} = 1$.
 Using Equation \eqref{DefB} we then obtain $\widetilde{B}(n,[\gamma],[*]) = 1 = \widetilde{B}(n,[*],[\gamma])$ for any $[\gamma] \in \pi_1(M)$.
 
 \item We may assume (additionally) that, for all commutators $[c] = [\tau] * [\gamma] * [\tau]^{-1} * [\gamma]^{-1}$ we have $\Lambda(\widehat{n},[c]) = 1$.
 
 Indeed, fix $[\tau], [\gamma] \in \pi_{1}(M)$, and set $[c] = [\tau] * [\gamma] * [\tau]^{-1} * [\gamma]^{-1}$.
 By step 1, we already know that if $[c]$ is conjugate to $[*]$, then $\Lambda(\widehat{n},[c]) = 1$, so we assume that $[c]$ is not conjugate to $[*]$.
 Since $\Lambda(\widehat{n}, \,\cdot\,)$ is a class function, we have \[\Lambda(\widehat{n}, [\tau]*[\gamma]*[\tau]^{-1}) = \Lambda(\widehat{n},[\gamma]),\] 
 so \eqref{DefB} yields
 \begin{eqnarray*}
 \Lambda(\widehat{n}, [\tau]*[\gamma]*[\tau]^{-1})\Lambda(\widehat{n}, [\gamma]^{-1}) &=& \Lambda(\widehat{n}, [c])
 B(n, [\tau]*[\gamma]*[\tau]^{-1}, [\gamma]^{-1})\\
 =\Lambda(\widehat{n}, [\gamma]) \Lambda(\widehat{n}, [\gamma]^{-1}) &=& \Lambda(\widehat{n}, [*]) B(n, [\gamma], [\gamma]^{-1}).
 \end{eqnarray*}
 Since $\Lambda(\widehat{n}, [*]) = 1$ by step 1, we find 
 \begin{equation}\label{eq:conjuneer}
 \Lambda(\widehat{n}, [c]) = \frac{B(n, [\gamma], [\gamma]^{-1})}{B(n, [\tau]*[\gamma]*[\tau]^{-1}, [\gamma]^{-1})}
\end{equation}
 on $U := U_{[c]} \cap U_{[\gamma]} \cap U_{[\gamma]^{-1}}\cap U_{[*]}$. 
Note that the right hand side depends on $\widehat{n}$ only through $n$, and is a smooth function of $n$.
We define $A(n, [\sigma])$ to be the r.h.s. of \eqref{eq:conjuneer} if $[\sigma]$ is conjugate to $[c]$, and $A(n, [\sigma]) = 1$ otherwise.
Because $[c]$ is not conjugate to $[*]$ we have $A(n, [*]) = 1$ and thus $\widetilde{B}(n,[\gamma],[*]) = B(n,[*],[\gamma]) = 1$.

Then $A$ is smooth in the first variable, it is a class function in the second variable, and after applying the trick we have
$\widetilde{\Lambda}(\widehat{n}, [c]) = 1$. In fact, since $\widetilde{\Lambda}(\widehat{n}, \,\cdot\,)$ is a class function, 
we even have $\widetilde{\Lambda}(\widehat{n}, [\sigma]) = 1$ for every $[\sigma]$ conjugate to $[c]$, whereas the value of 
$\widetilde{\Lambda}(\widehat{n}, [\sigma])$ remains unchanged if $[\sigma]$ is not conjugate to $[c]$.
(In particular, we still have $\widetilde{\Lambda}(\widehat{n}, [*]) = 1$.)
Repeating this process, we can achieve that $\widetilde{\Lambda}(\widehat{n}, [c]) = 1$ for every commutator. (Applying the trick only 
requires one to adapt the domain $U_{[\sigma]}$ within the conjugacy class of $[c]$, so every domain is intersected with at most 3 other open sets.)

\item We may assume (additionally) that $\Lambda(\widehat{n},[c]) = 1$ for all $[c] \in [\pi_{1}(M),\pi_{1}(M)]$.

Let $F_{m} \subseteq [\pi_1(M), \pi_1(M)]$ be the set of elements that can be written as a product 
of at most $m$ commutators.
Suppose by induction that $\Lambda(\widehat{n}, [c]) = 1$ for every $[c]\in F_m$. (The induction hypothesis is step 2.)
Writing $[c] \in F_{m+1}\setminus F_{m}$ as $[c] = [c_1]*[c_2]$ with $[c_1], [c_2] \in F_{m}$, we find 
\[
1 = \Lambda(\widehat{n}, [c_1])\Lambda(\widehat{n}, [c_2]) = \Lambda(\widehat{n}, [c])B(n, [c_1], [c_2]),
\]
so setting $A(n, [\sigma]) := B(n, [c_1], [c_2])$ for every $[\sigma]$ in the conjugacy class of $[c]$ and $A(n, [\sigma]) = 1$
otherwise, we find that $\widetilde{\Lambda}(\widehat{n}, [c]) = 1$. Since $\Lambda$ is a class function, we have $\widetilde{\Lambda}(\widehat{n}, [\sigma]) = 1$ on the conjugacy class of $[c]$, and since the sets $F_m$ are conjugation invariant, we still have $\widetilde{\Lambda}(\widehat{n}, [\sigma]) = 1$ for $[\sigma] \in F_m$.
Repeat this for every conjugacy class of $F_{m+1}\setminus F_m$ to find $\widetilde{\Lambda}(\widehat{n}, [c]) = 1$ on $F_{m+1}$.
By induction, there exists a choice of sections such that $\Lambda(n, [c]) = 1$ for every $[c] \in [\pi_1(M), \pi_1(M)]$.
(Note that step $m$ of the induction affects the domain $U_{[c]}$ only for $[c] \in F_{m+1}\setminus F_{m}$, so every domain
is clipped only finitely many times.)

 

\item We may assume (additionally) that, if $[[\gamma]]$ is of finite order $k$ in $H_{1}(M,\Z)$, 
then $\Lambda( \widehat{n}, [\gamma])^{k} = 1$.

Indeed, let $[\gamma]$ be as above.
Then $[\gamma]^k \in [\pi_1(M), \pi_1(M)]$, so $\Lambda(\widehat{n}, [\gamma]^k) = 1$ by step 3. 
Since 
\[
	\Lambda(\widehat{n}, [\gamma])\Lambda(\widehat{n}, [\gamma]^i) = B(n, [\gamma], [\gamma]^{i}) \Lambda(\widehat{n}, [\gamma]^{i+1})
\]
for all $i \in \{0, \ldots, k-1\}$, and since $\Lambda(\widehat{n}, [\gamma]^{k}) = 1$, we find inductively that 
\[
\Lambda(\widehat{n}, [\gamma])^k = \prod_{i=0}^{k-1}B(n, [\gamma], [\gamma]^i) =: \beta(n, [\gamma])
\]
on $\bigcap_{i=0}^{k} U_{[\gamma]^k}$. Since $\sigma_{\gamma}(\one) = \one$, we have 
$\Lambda(\one, [\gamma]) = 1$, so
$B(\one, [\gamma], [\tau]) = 1$ for all $[\gamma], [\tau] \in \pi_1(M)$. It follows that $\beta(\one, [\gamma]) = 1$. 
Let $W \subseteq \ab$ be an open neighbourhood of $0$ for which the exponential map $\exp \colon \ab  \supseteq W \rightarrow \Ab$ 
is a diffeomorphism onto its image.
By choosing $U\subseteq \bigcap_{i=0}^{k} U_{[\gamma]^k}$ 
sufficiently small, we have $\beta(U, [\gamma]) \subseteq \exp(W)$, yielding a smooth $k^{\mathrm{th}}$ root 
\[
A(n, [\gamma]):= \sqrt[k]{ \beta(n, [\gamma])}
\]
on $U$ with $A(\one, [\gamma]) = 1$. Since $\Lambda(\widehat{n}, [\gamma])^k = A(n, [\gamma])^k$, we find 
$\widetilde{\Lambda}(\widehat{n}, [\gamma])^k = 1$ after applying the trick.
\end{enumerate}
Using a local section $\sigma_{\comp{\gamma}} \colon N \supseteq U \rightarrow N^{\sharp}_{\comp{\gamma}}$ with these properties, we obtain a local trivialization 
\begin{equation*}
			\begin{tikzcd}
				N^{\sharp}_{\comp{\gamma}}\ar[d, "\pi"]& \pi^{-1}(U) \ar[d, "\pi"] \ar[l, hook', "\iota"] \ar[r, "\sim"]& U \times \Ab \ar[d, "\mathrm{pr_{U}}"]\\
				N & U\ar[l, hook', "\iota"] \ar[r, equal]&U
			\end{tikzcd}
		\end{equation*}
of the smooth principal $\Ab$-bundle $\pi \colon N^{\sharp}_{\comp{\gamma}} \rightarrow N$
over the identity neighbourhood $U \subseteq N$. Since ${\Lambda( \widehat{n}, [\gamma])^{k} = 1}$, 
the subgroup $\widehat{N}_{\comp{\gamma}} \subseteq N^{\sharp}_{\comp{\gamma}}$
corresponds to a $\Gamma_{\comp{\gamma}}$-equivariant subset of $U \times \Ab_{k}$ with respect to this local trivialization. Since $\Gamma_{\comp{\gamma}} \cap \Ab_0 = \Ab_k \cap \Ab_0$, 
the intersection of $\widehat{N}_{\comp{\gamma}}$ with the open identity neighbourhood $U \times \Ab_0$ of $N^{\sharp}_{\comp{\gamma}}$
is $U \times \Gamma_{\comp{\gamma}}$, and we conclude that $\widehat{N}_{\comp{\gamma}}$ is an embedded subgroup of~$N^{\sharp}_{\comp{\gamma}}$.
\end{proof}

Because $N$ acts trivially on $\pi_0(LM)$, it also acts trivially on $H_{1}(M, \mathbb{Z})$, which we assume to be finitely generated.
Pick generators $[\alpha_{i}] \in H_{1}(M, \mathbb{Z})$ and choose lifts $[\gamma_{i}] \in \pi_{1}(M)$.
Suppose that the generators are indexed by $i \in \{1, 2, ..., k\}$.
Taking the fibre product over $N$ of the Lie groups $\widehat{N}_{\comp{\gamma_{i}}}$ we obtain a Fr\'echet--Lie group
\begin{equation}\label{eq:defNhat}
\widehat{N}' := \widehat{N}_{\comp{\gamma_{1}}}\times_{N} \ldots \times_{N} \widehat{N}_{\comp{\gamma_{k}}}, 
\end{equation}
and a central Lie group extension
\begin{equation}
\prod_{i=1}^{k}\Gamma_{\comp{\gamma_{i}}} \rightarrow \widehat{N}' \rightarrow N
\end{equation}
of $N$ by the abelian Lie group $\Gamma := \prod_{i=1}^{k}\Gamma_{\comp{\gamma_{i}}}$. 
Since every diagram 
	\begin{equation}
		\begin{tikzcd}
			H^{1}(M,\Ab) \ar[r] \ar[d] & \widehat{N} \ar[r] \ar[d] & N \ar[d,equal] \\
			\Gamma_{\comp{\gamma_{i}}} \ar[r] & \widehat{N}_{\comp{\gamma_{i}}} \ar[r] & N
		\end{tikzcd}
	\end{equation}
	is commutative, the diagram 
	\begin{equation}\label{diag:niceEquivalence}
		\begin{tikzcd}
			H^{1}(M,\Ab) \ar[r] \ar[d] & \widehat{N} \ar[r] \ar[d] & N \ar[d,equal] \\
			\Gamma \ar[r] & \widehat{N}' \ar[r] & N
		\end{tikzcd}
	\end{equation}
	is commutative as well.
	The vertical arrow on the left is the product of the evaluation homomorphisms 
\begin{equation}\label{eq:evalLieGp}
	\prod_{i=1}^{k}\ev_{[[\gamma]]}\colon \quad
	\mathrm{Hom}(H_{1}(M,\mathbb{Z}), \Ab) \longrightarrow \Gamma = \prod_{i=1}^{n}\Gamma_{\comp{\gamma_{i}}},
\end{equation}
which is surjective by definition, and injective because a group homomorphism is completely determined by its values on a set of generators.
It follows that \eqref{diag:niceEquivalence} is an equivalence of central group extensions, so the Fr\'echet--Lie group structure can be transported from the bottom row to the top.	

Recall that every $\Gamma_{\comp{\gamma_{i}}}$ is an embedded Lie subgroup of $\Ab$. 
Indeed, $\Gamma_{\comp{\gamma_i}}\subseteq \Ab$ is an open subgroup 
if $[\gamma_i]\in H_1(M,\Z)$ is of infinite order, and a discrete subgroup if $[\gamma_i]$ is of finite order.
We may assume without loss of generality that the first $b_1$ generators $[\alpha_i]$ are of infinite order, where $b_1$ is the first Betti number of $M$, and the remaining generators are of finite order.
The Lie algebra of $\Gamma$ can then be identified with the abelian Lie algebra
\begin{equation*}
	\mathrm{Lie}(\Gamma) = \bigoplus_{i=1}^{b_1}\ab_{i},
\end{equation*}
and the Lie algebra cocycles \eqref{eq:Nicocycle} for $\widehat{N}_{\comp\gamma}$ combine to the Lie algebra cocycle 
\[
\Psi'(X,Y) =  \bigoplus_{i=1}^{b_1} \left(\oint_{\gamma_i} i_{X_M}i_{Y_M}\omega\right)
\]
on $\mathfrak{n}$ with values in $\bigoplus_{i=1}^{b_1}\ab_{i}$.

Differentiating \eqref{eq:evalLieGp}, the evaluation homomorphism
\begin{equation*}
\ev \colon H^1_{\mathrm{dR}}(M,\ab) \rightarrow \bigoplus_{i=1}^{b_1}\ab_i, \quad \ev([\theta]) := \bigoplus_{i=1}^{b_1} \oint_{\alpha_i}\theta
\end{equation*}
is an isomorphism of abelian Lie algebras.
The central Lie algebra extension $H^1_{\mathrm{dR}}(M,\ab ) \rightarrow \widehat{\lie{n}}_{\omega} \rightarrow \fn$ from \eqref{eq:AlgExtension2}
gives rise to the 2-cocycle 
\[
	\Psi(X,Y) = [i_{X_M}i_{Y_M}\omega]
\]
on $\fn$ with values in $H^1_{\mathrm{dR}}(M,\ab)$. Since $\ev \circ \Psi = \Psi'$, the central Lie algebra extension associated to the exact sequence $\Gamma \rightarrow \widehat{N}' \rightarrow N$ is isomorphic to \eqref{eq:AlgExtension2}.
This concludes the proof of the following theorem.

\begin{theorem}\label{Thm:centralcase}
Let $M$ be an orientable, connected, finite-dimensional  manifold with finitely generated $H_1(M,\Z)$, and let $\cG$ be a bundle gerbe on $M$ with curvature 
$\omega \in \Omega^3(M,\ab)$.
Let $N$ be be a Fr\'echet--Lie group that acts smoothly 
on $M$, and which preserves the class of $\cG$.  If $N$ acts trivially on $\pi_0(LM)$, then
\[
	H^1(M, \Ab) \rightarrow \widehat{N} \rightarrow N
\]
is a central extension of Fr\'echet--Lie groups, with corresponding central Lie algebra extension 
\[
	H^1_{\mathrm{dR}}(M,\ab) \rightarrow \widehat{\lie{n}}_{\omega} \rightarrow \fn
\]
described by \eqref{eq:AlgExtension2}, \eqref{eq:LAcentralextension}, \eqref{eq:definitionbracket}.
\end{theorem}

In equipping $\widehat{N}$ with the structure of Fr\'{e}chet manifold, we picked generators $[\alpha_{i}] \in H_{1}(M, \Z)$ and lifts $[\gamma_{i}] \in \pi_{1}(M)$.
In the sequel, we argue that the smooth structure on $\widehat{N}$ does not depend on any of these choices.

\begin{lemma}\label{lem:SmoothCommutators}
	If $\langle \gamma \rangle \subseteq LM$ is the connected component of a loop $\gamma$ such that $[\gamma ] \in [\pi_{1}(M),\pi_{1}(M)]$, then the map $\widehat{N} \rightarrow \widehat{N}_{\langle \gamma \rangle }$ is smooth.
\end{lemma}
\begin{proof}
	By Proposition \ref{prop:CoreTechnicalResult}, the Lie group $\widehat{N}_{\langle \gamma \rangle }$ is an extension of $N$ by $\Gamma_{\langle\gamma\rangle}$, where $\Gamma_{\langle \gamma \rangle}$ is the image of the evaluation map $\ev_{[[\gamma]]}: K \rightarrow \Ab$.
	The assumption that $[\gamma] \in [\pi_{1}(M),\pi_{1}(M)]$ implies that $\Gamma_{\langle \gamma \rangle} = \{ 1 \}$.
	In turn, this implies that the projection map $\widehat{N}_{\langle \gamma \rangle} \rightarrow N$ is an isomorphism.
	The commutativity of the diagram
	\begin{equation*}
		\begin{tikzcd}
			\widehat{N} \ar[rd] \ar[r] & \widehat{N}_{\langle \gamma \rangle} \ar[d] \\
			& N \ar[u, bend right, dotted]
		\end{tikzcd}
	\end{equation*}
	then tells us that the map $\widehat{N} \rightarrow \widehat{N}_{\langle \gamma \rangle}$ is smooth.
\end{proof}

\begin{lemma}\label{lem:SmoothProducts}
	Let $[\gamma_{1}],[\gamma_{2}] \in \pi_{1}(M)$, and suppose that the maps $\widehat{N} \rightarrow \widehat{N}_{\langle \gamma_{1} \rangle }$ and $\widehat{N} \rightarrow \widehat{N}_{\langle \gamma_{2} \rangle }$ are smooth.
	Then, the map $\widehat{N} \rightarrow \widehat{N}_{\langle \gamma_{1} \star \gamma_{2} \rangle }$ is smooth.
\end{lemma}
\begin{proof}
	By Proposition \ref{prop:CoreTechnicalResult}, $\widehat{N}_{\langle \gamma_{1} \star \gamma_{2} \rangle }$ is an embedded subgroup of $N^{\sharp}_{\langle \gamma_{1} \star \gamma_{2} \rangle }$.
	It thus suffices to prove that the induced map $\widehat{N} \rightarrow N^{\sharp}_{\langle \gamma_{1} \star \gamma_{2} \rangle }$ is smooth.
	By definition of the smooth structure on $N^{\sharp}_{\langle \gamma_{1} \star \gamma_{2} \rangle }$ (see Proposition \ref{Prop:vasteklasse}) it suffices to find an element $p \in \mc{T}(\mc{G})|_{\langle \gamma_{1} \star \gamma_{2} \rangle}$ such that the orbit map ${\widehat{N} \rightarrow \mc{T}(\mc{G})|_{ \langle \gamma_{1} \star \gamma_{2} \rangle }}$ is smooth.

	To this end, we choose representatives $\gamma_{1}$ and $\gamma_{2}$ of $[\gamma_{1}]$ and $[\gamma_{2}]$, respectively.
	We view $\gamma_{1}$ and $\gamma_{2}$ as paths, and we observe that $(\gamma_{1},*,\gamma_{2}) \in PM^{[3]}$, where $*$ is the constant path.
	Pick $p_{1} \in (\pr_{12}^{*}\mc{T}(\mc{G}))_{(\gamma_{1},*)}$ and $p_{2} \in \pr_{23}^{*}(\mc{T}(\mc{G}))_{(*,\gamma_{2})}$.
	Set $q = \mu(p_{1} \otimes p_{2}) \in (\pr_{13}^{*}\mc{T}(\mc{G}))_{(\gamma_{1},\gamma_{2})} = \mc{T}(\mc{G})_{\gamma_{1} \star \gamma_{2}}$.
	The compatibility of $\hat{n}$ with the fusion product $\mu$ then implies that the orbit map $\hat{n} \mapsto \hat{n}q$ factors as follows:
	\begin{equation*}
		\begin{tikzcd}
			\hat{N} \ar[r] & \pr_{12}^{*}\mc{T}(\mc{G})|_{\langle \gamma_{1} \rangle } \otimes \pr_{23}^{*}\mc{T}(\mc{G})|_{\langle \gamma_{2} \rangle } \ar[r, "\mu"] & \pr_{13}^{*}\mc{T}(\mc{G})|_{\langle \gamma_{1} \star \gamma_{2} \rangle} \ar[r] & \mc{T}(\mc{G})|_{\langle \gamma_{1} \star \gamma_{2} \rangle}.
		\end{tikzcd} 
	\end{equation*}
	Since all these maps are smooth in the diffeological sense, and since $\hat{N}$ and $\mc{T}(\mc{G})|_{\langle \gamma_{1} \star \gamma_{2} \rangle}$ 
	are Fr\'echet manifolds, their concatenation is a smooth map in the category of Fr\'echet manifolds.
\end{proof}

\begin{proposition}\label{prop:GeneratorIndependence}
	The manifold structure of $\widehat{N}$ does not depend on the choice of generators.
	Moreover, the map $\widehat{N} \rightarrow \widehat{N}_{\comp{\gamma}}$ is smooth for every connected component $\comp{\gamma} \subseteq LM$.
\end{proposition}
\begin{proof}
	Let $\langle \gamma \rangle \in \pi_{0}(LM)$ be arbitrary.
	Pick a representative $[\gamma] \in \pi_{1}(M)$ of $\langle \gamma \rangle$.
	Write $[[\gamma]] = [\gamma][\pi_{1}(M),\pi_{1}(M)] \in H_{1}(M,\Z)$ in terms of the generators $[[\gamma_{i}]]$, i.e. $[[\gamma]] = \sum_{i=1}^{j}n_i [[\gamma_i]]$.
	Then, the element $a = [\gamma_1]^{n_1} \cdots [\gamma_{j}]^{n_{j}}$ of $\pi_1(M)$ satisfies $[a] = [[\gamma]]$ in $H^1(M,\Z)$.

	It follows that there exists an element $c \in [\pi_{1}(M),\pi_{1}(M)]$ such that $ac = [\gamma]$.
	Now, repeated application of Lemma \ref{lem:SmoothProducts} implies that the map $\hat{N} \rightarrow \hat{N}_{\langle a \rangle }$ is smooth.
	Lemma \ref{lem:SmoothCommutators} implies that the map $\hat{N} \rightarrow \hat{N}_{\langle c \rangle }$ is smooth.
	A final application of Lemma \ref{lem:SmoothProducts} then tells us that the map $\hat{N} \rightarrow \hat{N}_{\langle a c \rangle } = \hat{N}_{\langle \gamma \rangle }$ is smooth.

	Given another set of generators $[\alpha_{i}'] \in H_{1}(M,\Z)$, and lifts $[\gamma'_{i}] \in \pi_{1}(M)$ we then see from the above that the map $\widehat{N} \rightarrow \widehat{N}_{\comp{\gamma_{1}'}} \times_{N} ... \times_{N} \widehat{N}_{\comp{\gamma_{k}'}}$ is smooth, and the result follows.
\end{proof}

From Theorem \ref{Thm:centralcase}, Theorem \hyperlink{target:A}{A} can easily be deduced:
\begin{proof}[Proof of Theorem A]
If $G$ is a Fr\'echet--Lie group which does not necessarily act trivially on $\pi_0(LM)$, then the subgroup $N\subseteq G$ that fixes $\pi_0(LM)$ is open, and hence 
a Fr\'echet--Lie subgroup of $G$. 
Since $K = H^1(M,\Ab)$ is central in $\widehat{N}$, the conjugation action of $\widehat{G}$ on $\widehat{N}$ factors through $G$.
 The action of $G$ on $K$ is smooth by Corollary~\ref{cor:ConjugationOnH1}, and the action of $G$ on $N$ is smooth because $N$ is a Lie subgroup.
It follows that the conjugation action of $\widehat{G}$ on $\widehat{N}$ is smooth, so that $\widehat{G}$ is a Lie group.
Since $N \subseteq G$ is open, both groups have the same Lie algebra.
This concludes the proof of Theorem~\hyperlink{target:A}{A}.
\end{proof}

\subsection{Functoriality of the Fr\'echet--Lie group extension}

The smooth structure on the abelian extension is functorial in the following sense.
Let \[\Pi = (\pi, f, \Lambda): \; \Big(G \curvearrowright (M,\mc{G}_{M})\Big) \rightarrow \Big(H \curvearrowright (N,\mc{G}_{N})\Big)\] be a morphism 
of group actions as defined in Section \ref{sec:functoriality}.
\begin{proposition}\label{prop:smoothfunctor}
	The map $\widehat{\Pi}: \widehat{G}_{\mc{G}_{M}} \rightarrow \widehat{H}_{\mc{G}_{N}}$ defined in Equation \eqref{eq:functorMorphism} is a morphism in the category of Fr\'{e}chet-Lie central extensions.
\end{proposition}
\begin{proof}
 By Proposition \ref{prop:DiscreteFunctoriality} it suffices to prove that $\widehat{\Pi}$ is smooth.

Let $N_{G} \subseteq G$  and $N_{H} \subseteq H$ be the normal subgroups that act trivially on $\pi_0(LM)$, and $\pi_{0}(LN)$ respectively.
Let $\gamma_{N} \in LN$ be a loop, and set $\gamma_{M} = Lf(\gamma_{N}) \in LM$.
Fix $q \in \tgress(\mc{G}_{N})_{\gamma}$, and set $p = Lf_{*} \tgress(\Lambda)q \in \tgress(\mc{G}_{M})_{\gamma_{M}}$.
We then make use of the following identification of principal $\Ab$-bundles over $N_{G}$
\begin{align*}
	N^{\sharp}_{G,\gamma_{M}} &\xrightarrow{\sim} \alpha^{*}_{\gamma_{M}} \tgress(\mc{G}_{M}), & (g,\nu) &\mapsto (g, \nu(p)).
\end{align*}
Indeed, this identification is how $N^{\sharp}_{G,\gamma_{M}}$ is equipped with a smooth structure in the first place (see Proposition \ref{Prop:vasteklasse} and \cite{NV03}).

We have a commutative diagram
\begin{equation*}
	\begin{tikzcd}
		N^{\sharp}_{G,\gamma_{M}} \ar[r] \ar[d] &G \ar[r,"\pi"] \ar[d, "\alpha_{\gamma_{M}}"] & H \ar[d, "\alpha_{\gamma_{N}}"] & N^{\sharp}_{H,\gamma_{N}} \ar[l] \ar[d] \\
		\tgress(\mc{G}_{M}) \ar[r] & LM & LN \ar[l, "Lf"] & \tgress(\mc{G}_{N}). \ar[l]
	\end{tikzcd}
\end{equation*}
By way of the isomorphism $\tgress(\Lambda): \tgress(\mc{G}_{N}) \rightarrow Lf^{*} \mc{T}(\mc{G}_{M})$, the diagram above allows us to identify $N^{\sharp}_{G, \gamma_{M}} \cong \pi^{*} N^{\sharp}_{H,\gamma_{N}}$, which induces a smooth pushforward map $\pi_{*}:N^{\sharp}_{G,\gamma_{M}} \rightarrow N^{\sharp}_{H,\gamma_{N}}$.
Identifying $N^{\sharp}_{H,\gamma_{N}}$ with $\alpha^{*}_{\gamma_{N}}\tgress(\mc{G}_{N})$, one obtains a smooth map
\begin{equation*}
	\pi_{*}:N^{\sharp}_{G,\gamma_{M}} \rightarrow \alpha^{*}_{\gamma_{N}}\tgress(\mc{G}_{N}), \quad (g,\nu) \mapsto (\pi(g),\tgress(\Lambda)^{-1}(\pi(g)\gamma_{N},\nu(p))).
\end{equation*}

We claim that the restriction of the map $\pi_{*}$ to $\widehat{N}_{G, \gamma_{M}}$ is the restriction of $\widehat{\Pi}$ to $\widehat{N}_{G, \gamma_{M}}$.
To compare $\pi_{*}$ to $\widehat{\Pi}$ as defined in Equation \eqref{eq:functorMorphism} we make use of Proposition \ref{prop:GroupTransgression}.
Our goal is to prove that $\pi_{*} \tgress_{L}(g,A) = \tgress_{L} \widehat{\Pi}(g,A)$.
Indeed
\begin{align*}
	\pi_{*} \tgress_{L}(g,A) &= \pi_{*} (g, (L\rho_{g})_{*} \circ \tgress(A)^{-1}) \\
	&= (\pi(g), \tgress(\Lambda)^{-1} (\pi(g)\gamma_{N}, (L\rho_{g})_{*} \circ \tgress(A)^{-1} p)) \\
	&= (\pi(g), \tgress(\Lambda)^{-1} (\pi(g)\gamma_{N}, (L\rho_{g})_{*} \circ \tgress(A)^{-1} \circ Lf_{*} \circ \tgress(\Lambda) q)) \\
	&= (\pi(g), \tgress(\Lambda)^{-1} (\pi(g)\gamma_{N}, Lf_{*} \circ (L\rho_{\pi(g)})_{*}  \circ Lf^{*} \tgress(A)^{-1} \circ \tgress(\Lambda) q)) \\
	&= (\pi(g), \tgress(\Lambda)^{-1} \circ (L\rho_{\pi(g)})_{*}  \circ Lf^{*} \tgress(A)^{-1} \circ \tgress(\Lambda) q)
\end{align*}

On the other hand
\begin{align*}
	\tgress_{L} \widehat{\Pi}(g,A) &= \tgress_{L} (\pi(g), \Lambda^{-1} \circ f^{*}A \circ \rho_{\pi(g)}^{*} \Lambda)\\
	& = (\pi(g), (L\rho_{\pi(g)})_{*} \circ \tgress(\Lambda^{-1} \circ f^{*}A \circ \rho_{\pi(g)}^{*} \Lambda)^{-1}) \\
	&= (\pi(g), ((L\rho_{\pi(g)})_{*} \circ L\rho_{\pi(g)}^{*} \tgress(\Lambda)^{-1} \circ Lf^{*} \tgress(A)^{-1} \circ  \tgress(\Lambda)) \\
	&= (\pi(g), \tgress(\Lambda)^{-1} \circ (L\rho_{\pi(g)})_{*} \circ Lf^{*} \tgress(A)^{-1} \circ \tgress(\Lambda)).
\end{align*}

Since $\widehat{N}_{H, \gamma_{N}} \subseteq N^{\sharp}_{H,\gamma_{N}}$ is an embedded subbundle, and since the image of $\widehat{N}_{G,\gamma_{M}} \subseteq N^{\sharp}_{G, \gamma_{M}}$ under the pushforward map is contained in $\widehat{N}_{H, \gamma_{N}} \subseteq N^{\sharp}_{H, \gamma_{N}}$, the pushforward $\pi_{*}$ restricts to a smooth map $\pi_{*} \colon \widehat{N}_{G, \gamma_{M}} \rightarrow \widehat{N}_{H, \gamma_{N}}$ of principal bundles.

Now, suppose that we have generators $\alpha_{i}  \in H_{1}(N,\Z)$ and loops $\gamma_{i} \in LN$ such that $[[\gamma_{i}]] = \alpha_{i}$.
Set $\gamma_{i}' = Lf (\gamma_{i}) \in LM$.
We then have a commutative diagram
\begin{equation*}
	\begin{tikzcd}
		\widehat{N}_{G} \ar[rr,"\widehat{\Pi}"] \ar[rd] & & \widehat{N}_{H} = \widehat{N}_{H,\gamma_{1}} \times_{N_{H}} \ldots \times_{N_{H}} \widehat{N}_{H,\gamma_{k}} \\
		& \widehat{N}_{G, \gamma_{1}'} \times_{N_{G}} \ldots \times_{N_{G}} \widehat{N}_{G,\gamma_{k}'} \ar[ru, "\pi_{*}"'] &
	\end{tikzcd}
\end{equation*}
It thus follows that $\widehat{\Pi}$ is a homomorphism of Lie groups, which restricts to a smooth map on the connected component of the identity, and is thus smooth on the the entire Lie group.
\end{proof}

Specializing to the situation that $\mc{G}_{M} = \mc{G}_{N}$, and $M=N$, and $f= \1_{M}$, (cf.~Corollary \ref{functorialGp}) we obtain the following result.

\begin{corollary}\label{cor:smoothfunctorGp}
Suppose that $G$ acts smoothly on $M$ in a way that preserves the isomorphism class of a gerbe $\cG$.
If $\pi \colon H \rightarrow G$ is a smooth homomorphism of Fr\'echet--Lie groups,
then the induced morphism $\widehat{\pi} \colon \widehat{H}_{\cG} \rightarrow \widehat{G}_{\cG}$ of abelian extensions is smooth as well.
Moreover, if $\pi \colon H \hookrightarrow G$ is an inclusion of an embedded subgroup, then so is
$\widehat{\pi} \colon \widehat{H}_{\cG} \hookrightarrow \widehat{G}_{\cG}$.
\end{corollary}
\begin{proof}
	Apart from the claim that $\widehat{\pi}$ is an embedding if $\pi$ is so, this follows directly from Proposition \ref{prop:smoothfunctor}.

	Now, assume that $\pi: H \hookrightarrow G$ is an inclusion of an embedded Fr\'{e}chet--Lie group.
	We use the notation from the proof of Proposition \ref{prop:smoothfunctor}, except that we drop the subscript $M$.
	Let $n \in H \cap N_{G}$ be arbitrary, and let $\phi: G \supseteq U_{G} \xrightarrow{\sim} V_{G}$ be a slice chart for $H\subseteq G$ around $n$.
	Let $\gamma \in LM$, and let $\chi: \tgress(\mc{G})|_{W} \rightarrow W \times \Ab$ be a local trivialization around $n \cdot \gamma \in W \subseteq LM$.
	Denote by $\chi_{2}: \tgress(\mc{G})|_{W} \rightarrow \Ab$ the second component of $\chi$.
	We then obtain a local trivialization $(\mathrm{Id}_{W'},\chi_{2}): \alpha_{\gamma}^{*} \tgress(\mc{G})|_{W'} \rightarrow W' \times \Ab$, where $W' = \alpha_{\gamma}^{-1}(W) \cap U_{G} \subseteq G$.
	Using the identification $N_{G,\gamma}^{\sharp} \simeq \alpha^{*}_{\gamma}\tgress(\mc{G})$ we obtain a slice chart $(\phi,\chi_{2})$ for $N^{\sharp}_{H,\gamma} \subseteq N^{\sharp}_{G,\gamma}$.

	So if the homology class $[[\gamma]]$ is of infinite order in $H^1(M,\Z)$, then $\widehat{N}_{H, \gamma} = N^{\sharp}_{H, \gamma}$ is an embedded Lie subgroup of $\widehat{N}_{G, \gamma} = N^{\sharp}_{G, \gamma}$. If $[[\gamma]]$ is of finite order, then $\widehat{N}_{H,\gamma} \subseteq \widehat{N}_{G, \gamma}$ is an embedded submanifold because $\widehat{N}_{G,\gamma}$ is a discrete cover of $N_{G,\gamma}$.
	It follows that the fibre products $\widehat{N}_{H} \subseteq \widehat{N}_{G}$ from \eqref{eq:defNhat} are embedded subgroups, which implies the desired result because these are open in $\widehat{H}_{\cG}$ and $\widehat{G}_{\cG}$, respectively.
\end{proof}

\section{Exact volume-preserving diffeomorphisms of a 3-manifold}\label{sec:volumePreserving}

Let $M$ be a compact, connected, orientable manifold of dimension 3, with volume form $\mu \in \Omega^3(M)$. 
A vector field $X$ on $M$ is divergence free (with respect to $\mu$) if $L_{X}\mu = 0$, which is the case if and only if $i_{X}\mu$ is closed.
Denote the space of divergence free vector fields by $\X(M,\mu)$.
A divergence free vector field $X$ is called \emph{exact} divergence free if $i_{X}\mu$ is exact, and we denote the Lie algebra of exact divergence free vector fields 
by $\X_{\ex}(M,\mu)$.
Since $\X_{\ex}(M,\mu)$ is the kernel of the flux homomorphism 
\begin{equation}\label{eq:FluxHomomorphism}
	\flux_{\mu} \colon \X(M,\mu) \rightarrow H^{2}_{\mathrm{dR}}(M), \quad X \mapsto [i_{X}\mu],
\end{equation}
(cf.\ e.g.~\cite[\S3.2]{DJNV21}), it is an ideal in $\X(M,\mu)$ of finite codimension. 

The Lie algebra extension $\widehat{\fg}_{\mu}$ (see Equation \ref{eq:LAcentralextension}), for $\fg = \X_{\ex}(M,\mu)$ is simply 
$\ol{\Omega}{}^{1}(M) = \Omega^{1}(M)/d\Omega^{0}(M)$.
Indeed, since $\mu$ is nondegenerate, an exact volume-preserving vector field with $i_{X}\mu = d\psi_X$ 
is uniquely determined by the potential $\psi_{X}$, and $X$
determines $\psi_{X}$ up to closed forms.
The Lie bracket on $\ol{\Omega}{}^1(M)$ is given by
\begin{equation}
	[\ol{\psi}{}_{X}, \ol{\psi}{}_{Y}] = \ol{i_{X}i_{Y}\mu},
\end{equation}
and the projection $\pi \colon \ol{\Omega}{}^{1}(M) \rightarrow \X_{\ex}(M,\mu)$ satisfies $\pi(\ol{\psi}{}_{X}) = X$ if and only if 
$i_{X}\mu = d\psi_{X}$.

It is conjectured  \cite{Ro95} that
the central Lie algebra extension 
\begin{equation}\label{eq:LAexactextension}
H^{1}_{\dR}(M) \rightarrow \ol{\Omega}{}^{1}(M) \rightarrow \X_{\ex}(M,\mu)
\end{equation} 
is \emph{universal}, and one of us is currently working on a rigorous proof 
\cite{LeoCornelia}.
As a first application of our results, we construct a central Lie group extension that integrates \eqref{eq:LAexactextension}.

\subsection{Differential characters}\label{sec:DifferentialChars}

Since $M$ is compact, we may rescale the volume form $\mu$ to be integral, $\int_{M}\mu \in \N$.
Recall that isomorphism classes of $\U(1)$-bundle gerbes with connection are classified by differential characters $h \in \widehat{H}^2(M, \U(1))$, and that the differential cohomology group $\widehat{H}^2(M, \U(1))$ sits at the centre of the following commutative diagram with exact rows and columns \cite[p.~20]{BB14}:
	\begin{equation}\label{diag:niceGrid}
		\begin{tikzcd}[column sep=small]
			 &0 \ar[d] &0 \ar[d] &0 \ar[d] \\
			0 \ar[r]& J^2(M) \ar[r] \ar[d]& \Omega^2(M)/\Omega^2(M)_{\Z} \ar[r, "d"] \ar[d]&d\Omega^2(M) \ar[d] \ar[r]& 0\\
			0 \ar[r]&H^2(M,\U(1)) \ar[r] \ar[d]& \widehat{H}^2(M, \U(1)) \ar[r, "\mathrm{curv}"] \ar[d, "\mathrm{Ch}"]& \Omega^3(M)_{\Z} \ar[d]\ar[r]& 0\\
			0 \ar[r]& \mathrm{Ext}(H_2(M,\Z),\Z) \ar[r] \ar[d]&H^3(M,\Z) \ar[r] \ar[d]&\Hom(H_3(M,\Z),\Z)\ar[d]\ar[r] & 0\\
			&0 & 0 & 0. &
		\end{tikzcd}
	\end{equation}
In the middle row of this diagram, 
the curvature $\curv(h)$ of the differential character $h$ corresponds to the curvature of the corresponding bundle gerbe with connection.
Since $\mu$ is an integral $3$-form, 
there exist bundle gerbes $\cG$ with curvature $\mu$, and the set of bundle gerbes with this property 
is a torsor over $H^2(M,\U(1)) = \Hom(H_2(M,\Z), \U(1))$.

\subsection{An abelian extension of the group of symmetries of a differential character}
The group $\Diff(M,\mu)$ of volume-preserving diffeomorphisms is a Fr\'echet--Lie group \cite[Theorem 2.5.3]{Hamilton82}, 
and its Lie algebra is the Lie algebra $\X(M,\mu)$ of divergence-free vector fields.
Let $h \in \widehat{H}^{2}(M,\U(1))$ be a differential character with $\curv(h) = \mu$.
We then consider the map 
\begin{equation*}
	\Diff(M,\mu) \rightarrow \widehat{H}^{2}(M,\U(1)), \phi \mapsto (\phi^{-1})^*h - h.
\end{equation*}
Since $\Diff(M,\mu)$ is orientation preserving, it acts trivially on $H^3(M,\Z)$, so that the image of the map $\phi \mapsto (\phi^{-1})^{*}h - h$ lies in the kernel of $\mathrm{Ch}$.
It then follows from Diagram~\eqref{diag:niceGrid} that $(\phi^{-1})^*h - h \in J^{2}(M)$.
We thus define the Flux cocycle to be the map
\[
\Flux_{h} \colon \Diff(M,\mu) \rightarrow J^2(M)\,,\quad 
\Flux_{h}(\phi) := (\phi^{-1})^*h - h
\]
that takes values in the Jacobian torus $J^2(M) = H^2(M,\R)/H^2(M,\Z)_{\R}$, cf.~\cite[\S3.1]{DJNV21}.
The kernel of the flux cocycle is the group $\Diff(M,h)$ of all diffeomorphisms that preserve the differential 
character $h$, and hence the isomorphism class $[\cG]$ of the bundle gerbe.
Since this is a Fr\'echet--Lie group by \cite[Proposition 3.8]{DJNV21}, we conclude from Theorem~\ref{thm:A} that the group 
\[
	\widehat{\Diff}(M,h)_{\cG} = \{(\phi, A) \,; \phi \in \Diff(M,h), A \in \Mor h_1(\phi^*\cG, \cG)\}
\]
is an abelian extension of Fr\'echet--Lie groups, and the following diagram is exact:
	\begin{equation}\label{diag:firstExtension}
		\begin{tikzcd}[row sep = scriptsize]
			 &1 \ar[d]\\
			&H^1(M,\U(1))\ar[d]\\
			&\widehat{\Diff}(M,h)_{\cG}\ar[d]\\
			1 \ar[r]&\Diff(M,h) \ar[d]\ar[r] &\Diff(M,\mu) \ar[r, "\Flux_{h}"]& J^2(M) \ar[r]& 1.\\
			&1
		\end{tikzcd}
	\end{equation}
\subsection{Central extension of the exact volume-preserving diffeomorphisms}	
	
The flux cocycle $\Flux_{h}$ on $\Diff(M,\mu)$ is in general \emph{not} a group homomorphism. 
However, since the connected component
$\Diff(M,\mu)_0$ acts trivially on $H^2(M,\U(1))$, $\Flux_{h}$ restricts to a group homomorphism  
\[\Flux_{h} \colon \Diff(M,\mu)_{0} \rightarrow J^2(M)\]
on the connected component of unity.
The flux homomorphism $\Flux_{h}$ on $\Diff(M,\mu)_{0}$ is independent of the choice of differential character, 
and depends only on the integral volume form $\mu$, so that we may write $\Flux_{\mu} := \Flux_{h}$, cf.~\cite[\S3.1]{DJNV21}, see also \cite{Calabi1970,Banyaga1978,NV03}. 
Its kernel is $\Diff_{\ex}(M,\mu)$, the Fr\'echet--Lie group of \emph{exact} volume-preserving diffeomorphisms.
In the resulting exact sequence of Fr\'echet--Lie groups
	\begin{equation}\label{eq:exactnicediagram}
		\begin{tikzcd}[row sep = scriptsize]
			 &1 \ar[d]\\
			&H^1(M,\U(1))\ar[d]\\
			&\widehat{\Diff}_{\ex}(M,\mu)_{\cG}\ar[d]\\
			1 \ar[r]&\Diff_{\ex}(M,\mu) \ar[d]\ar[r] &\Diff(M,\mu)_0 \ar[r, "\Flux_{\mu}"]& J^2(M) \ar[r]& 1,\\
			&1
		\end{tikzcd}
	\end{equation}
all arrows are homomorphisms of Fr\'echet--Lie groups.
The vertical column is a \emph{central} extension because the connected Lie group $\Diff_{\ex}(M,\mu)$ acts trivially on $\pi_0(LM)$.

Since $\Diff_{\ex}(M,\mu)$ is open in $\Diff(M,h)$, the two diagrams \eqref{diag:firstExtension} and \eqref{eq:exactnicediagram}
give rise to the same exact diagram of Fr\'echet--Lie algebras:
	\begin{equation}\label{eq:nicediagram}
		\begin{tikzcd}[row sep = scriptsize]
			 &0 \ar[d]\\
			&H^1_{\dR}(M)\ar[d]\\
			&\ol{\Omega}{}^{1}(M)\ar[d]\\
			0 \ar[r]&\X_{\ex}(M,\mu) \ar[d]\ar[r] &\X(M,\mu) \ar[r, "\flux_{\mu}"]& H^2_{\dR}(M) \ar[r]& 0.\\
			&0
		\end{tikzcd}
	\end{equation}
In particular, we have obtained a central extension 
\begin{equation}\label{eq:CentExtDiffEx}
1\rightarrow H^1(M,\U(1)) \rightarrow \widehat{\Diff}_{\ex}(M,\mu)_{\cG} \rightarrow \Diff_{\ex}(M,\mu) \rightarrow 1
\end{equation}
of Fr\'echet--Lie groups that integrates the central extension 
\begin{equation}\label{eq:CentExtExDiv}
0 \rightarrow H^1_{\dR}(M) \rightarrow \ol{\Omega}{}^{1}(M) \rightarrow \X_{\ex}(M,\mu) \rightarrow 0
\end{equation}
of Fr\'echet--Lie algebras.

\label{Rk:universal3d}
In \cite{Ro95}, it is conjectured that this central extension
is universal, and we expect to prove this in the work in progress \cite{LeoCornelia}. 
If this is indeed the case, then the Recognition Theorem in \cite{Neeb2002UCE} (cf.\ also \cite{Neeb2002CE, Ne04}) provides
conditions on the first and second homotopy group of $\Diff_{\ex}(M,\mu)$ under which (a discrete cover of)
$\widehat{\Diff}_{\ex}(M,h)_{\cG}$ is the \emph{universal} central extension of the exact volume-preserving diffeomorphism group.
It may be possible to check these conditions for Riemannian 3-manifolds that satisfy the generalized Smale conjecture,
but we will return to this matter in later work.
In the following, we will instead pursue a similar strategy for the Lie group of exact volume-preserving diffeomorphisms of a 2-manifold $\Sigma$.
Since a volume form $\sigma$ on $\Sigma$ is the same as a symplectic form, $\Diff_{\ex}(\Sigma, \sigma) = \Ham(\Sigma)$ is the group of 
hamiltonian diffeomorphisms. In this setting a universal Lie algebra extension is available \cite{JV15}, and 
the topology is well understood \cite{Smale1959, EarleEells1969, Gramain1973}.  
In Section~\ref{sec:quant} we first construct a central extension of the quantomorphism group which is weakly universal at the Lie algebra level, 
and then use this to construct a universal central extension of $\Ham(\Sigma)$ in Section~\ref{sec:UniversalHamilton}.
	
\section{Central extension of the quantomorphism group of a surface}\label{sec:quant}
Let $\Sigma$ be a 2-dimensional, compact, connected surface with integral symplectic form $\sigma \in \Omega^2(\Sigma)$. 
Then $(\Sigma,\sigma)$ is prequantizable. Let $\pi \colon P \rightarrow \Sigma$ be a principal $\U(1)$-bundle over $\Sigma$, and let
$\theta \in \Omega^1(P, \R)$ be a connection of curvature $\sigma$. 
The group $\Aut(P,\theta)$ of connection-preserving automorphisms of $P \rightarrow \Sigma$ is called
the \emph{quantomorphism group}. It is a Fr\'echet--Lie group by~\cite[VIII.4]{Om74}, \cite{RatiuSchmid1981}. 
The image of the connected identity component $\Aut_{0}(P,\theta)$ in $\Diff(\Sigma,\sigma)$ is the \emph{Hamiltonian diffeomorphism group}
$\mathrm{Ham}(\Sigma)$, yielding a central extension
\begin{equation}\label{eq:KKS2dGp}
1 \rightarrow \U(1) \rightarrow \Aut_{0}(P,\theta) \rightarrow \mathrm{Ham}(\Sigma) \rightarrow 1
\end{equation}
of Fr\'echet--Lie groups. 
Using Theorem~\ref{Thm:centralcase} and ideas and results from \cite{JV15, JV18}, we construct a central Lie group extension 
\begin{equation}
	H^1(P,\U(1)) \rightarrow \widehat{\Aut}_{0}(P,\theta)_{\cG} \rightarrow \Aut_{0}(P,\theta)
\end{equation}
of the quantomorphism group, where $\mc{G}$ is a bundle gerbe with curvature $\theta \wedge d \theta/ (2\pi)$.
We prove that the corresponding Lie algebra extension is \emph{versal}.
This means that it covers every other continuous, linearly split central extension of Lie algebras, but the covering map is not
necessarily unique.
In the context of the quantomorphism group, this is the best one could hope for.
\emph{Universal} extensions (where the covering map is unique) exist if and only if the Lie algebra is perfect, and this is not the case for the Lie algebra of $\Aut_{0}(P,\theta)$.

\subsection{Extensions of Lie algebras}

Recall that a vector field $X \in \X(\Sigma)$ is hamiltonian if 
there exists a function $f_{X} \in C^{\infty}(\Sigma)$ such that
$i_{X}\sigma = -df_{X}$. 
We denote the Fr\'echet--Lie algebra of Hamiltonian vector fields by $\X_{\ham}(\Sigma)$.
Since $\sigma$ is nondegenerate, the function $f_X$ determines $X$
uniquely, and, conversely, $f_{X}$ is determined by $X$ up to an additive constant.
This yields the Kostant-Kirillov-Souriau (KKS) extension 
\begin{equation}\label{eq:KKS2dLA}
	0 \rightarrow \R \rightarrow C^{\infty}(\Sigma) \rightarrow \X_{\ham}(\Sigma) \rightarrow 0
\end{equation}
of the Lie algebra $\X_{\ham}(\Sigma)$ of hamiltonian vector fields, where $C^{\infty}(\Sigma)$ is endowed with 
the Poisson Lie bracket
\begin{equation}
	\{f_{X}, f_{Y}\} = i_{Y}i_{X}\sigma.
\end{equation}
Note that since $\Sigma$ is compact, \eqref{eq:KKS2dLA} is split by the Lie algebra homomorphism 
\[
	s \colon C^{\infty}(\Sigma) \rightarrow \R, \quad f \mapsto \int_{\Sigma} f\sigma.
\]
(This is a Lie algebra homomorphism because $\{f,g\}\sigma = d(fdg)$ is exact.) 
We can thus identify $\X_{\ham}(\Sigma)$ with the Lie algebra $C^{\infty}_{0}(\Sigma)$
of functions that integrate to zero against the Liouville measure.
 
The KKS extension \eqref{eq:KKS2dLA} is the exact sequence of Lie algebras corresponding to $\eqref{eq:KKS2dGp}$.
Indeed, given a function $f \in C^{\infty}(\Sigma)$, there exists a unique vector field $\zeta_{f} \in \lie{X}(P)$ determined by the equations $i_{\zeta_{f}}\theta = \pi^*f$ and $i_{\zeta_{f}} \pi^*\sigma = -\pi^*df$.
If these equations are satisfied, then $L_{\zeta_{f}}\theta = 0$ by Cartan's formula, so that $\zeta_{f}$ lies in the Lie algebra of $\Aut_{0}(P,\theta)$.
Conversely, an infinitesimal automorphism of $(P,\theta)$ determines a Hamilton function by the same formulas.
In the following, we will frequently identify $f \in C^{\infty}(\Sigma)$ with $\pi^*f \in C^{\infty}(P)$ if no confusion can arise.

The 3-dimensional total space $P$ of the prequantum bundle carries the nondegenerate volume form $\theta \wedge d\theta$, 
and its normalization
\[
	\mu = \frac{1}{2\pi}\theta \wedge d\theta
\] 
has integral periods.
The flux homomorphism $\flux_{\mu} \colon \X_{\ex}(P,\mu) \rightarrow H^2_{\dR}(P)$ from Equation~\eqref{eq:FluxHomomorphism} vanishes on infinitesimal automorphisms of $(P,\theta)$.
Indeed, 
\be\label{eq:InfQuantomorphismIsExact}
	\flux_{\mu}(\zeta_{f}) = [i_{\zeta_{f}}\mu] = \frac{1}{2\pi} [(fd\theta - df \wedge \theta)] = \frac{1}{\pi}[fd\theta]
\ee
manifestly vanishes for $f = 1$, and it vanishes on the subalgebra $C_{0}^{\infty}(\Sigma)$ of functions that 
integrate to zero because for any $f_0 \in C^{\infty}_{0}(\Sigma)$, $fd\theta = \pi^*(f\sigma)$ is the pullback of an exact 2-form (cf.~\cite[Proposition 4.2]{JV18}).

The Lie algebra extension $\widehat{\fg}_{\mu}$ of $\fg = C^{\infty}(\Sigma)$ by $H^1_{\dR}(P)$ defined in \eqref{eq:AlgExtension2} and \eqref{eq:LAcentralextension} is therefore
\[
	\widehat{C^{\infty}}(\Sigma)_{\mu}  = \Big\{\ol{\psi}_{f} \in \ol{\Omega}{}^1(P)\,;\, d\psi_{f} = \frac{1}{2\pi}(fd\theta - df\wedge \theta) \text{ for some } f\in C^{\infty}(\Sigma)\Big\},
\]
with projection $\widehat{C^{\infty}}(\Sigma)_{\mu} \rightarrow C^{\infty}(\Sigma)$ given by $\ol{\psi}_{f} \mapsto f$, and Lie bracket 
\[[\ol{\psi}_{f}, \ol{\psi}_{g}] = \ol{i_{\zeta_{f}} i_{\zeta_{g}} \mu} = \frac{1}{2\pi}\left(\ol{fdg - gdf} -\ol{\{f,g\} \theta}\right). \]
Using $\{f,g\}\sigma = df\wedge dg$, one readily checks that $d\ol{\psi}_{\{f,g\}} = d[\ol{\psi}_{f},\ol{\psi}_{g}]$, so 
$\ol{\psi_{f}} \mapsto f$ is indeed a Lie algebra homomorphism with kernel $H^1_{\dR}(P)$.
This gives rise to the following exact diagram of Fr\'echet--Lie algebras:
	\begin{equation}\label{eq:anothernicediagram}
		\begin{tikzcd}[row sep = scriptsize]
			 & &0 \ar[d]\\
			& &H^1_{\dR}(P)\ar[d]\\
			& &\widehat{C^{\infty}}(\Sigma)_{\mu}\ar[d]\\
			0 \ar[r]&\R \ar[r]&C^{\infty}(\Sigma) \ar[d]\ar[r] &\X_{\ham}(\Sigma) \ar[r]& 0.\\
			& &0
		\end{tikzcd}
	\end{equation}
\subsection{Extensions of Lie groups}
Let $h \in \widehat{H}^2(P,\U(1))$ be a differential character with curvature $\curv(h) = \mu$, and let $\cG$ be a bundle gerbe on $P$ realizing $h$.
Then the Lie group $\Aut_{0}(P,\theta)$ preserves the isomorphism class of~$\cG$.
Indeed, 
since $[i_{\zeta_{f}}\mu] = 0$ by the discussion following Equation~\eqref{eq:InfQuantomorphismIsExact},
the infinitesimal flux homomorphism 
$\flux_{\mu} \colon \X(P,\mu) \rightarrow H^2_{\dR}(P)$ vanishes on $C^{\infty}(\Sigma)$.
Since $\Aut_{0}(P,\theta)$ is a connected Lie subgroup of $\Diff_{0}(P,\mu)$, it is in the kernel of the flux homomorphism $\Flux_{h} \colon \Diff_{0}(P,\mu) \rightarrow J^2(P)$, so $\phi^*h=h$ for all $\phi \in \Aut_{0}(P,\theta)$.
Corresponding to the exact diagram~\eqref{eq:anothernicediagram}, we thus obtain 
an exact diagram 
of Fr\'echet--Lie groups
	\begin{equation}\label{eq:anothernicediagramGroups}
		\begin{tikzcd}[row sep = scriptsize]
			 & &0 \ar[d]\\
			& &H^1(P,\U(1))\ar[d]\\
			& &\widehat{\Aut}_{0}(P,\theta)_{\cG}\ar[d]\\
			0 \ar[r]&\U(1) \ar[r]&\Aut_{0}(P,\theta) \ar[d]\ar[r] &\mathrm{Ham}(\Sigma) \ar[r]& 0,\\
			& &0
		\end{tikzcd}
	\end{equation}
where the vertical column is a central extension of Fr\'echet--Lie groups by Theorem~\ref{Thm:centralcase}. 

\subsection{Versal and universal Lie algebra extensions}

Since $\X_{\ham}(\Sigma) \simeq C^{\infty}_{0}(\Sigma)$ is perfect \cite{ALDM74}, it has a universal central extension. 
It is given \cite[Theorem 5.6]{JV15} by 
\begin{equation}\label{eq:UnivCE}
H^1_{\dR}(\Sigma) \rightarrow \ol{\Omega}{}^{1}(\Sigma) \rightarrow C^{\infty}_{0}(\Sigma),
\end{equation}
where $\ol{\gamma}_{f} \in \ol{\Omega}{}^{1}(\Sigma)$ projects to the unique $f\in C^{\infty}_{0}(\Sigma)$ such that $d\ol{\gamma}_{f} = f\sigma$, and the Lie bracket on $\ol{\Omega}{}^{1}(\Sigma)$ is given by \cite[Rk.~5.3]{JV15}
\begin{equation*}
	[\ol{\gamma}_{f}, \ol{\gamma}_{g}] = {\textstyle\frac{1}{2}}\ol{fdg - gdf}.
\end{equation*}
We define an injective Lie algebra homomorphism 
\[
J \colon \ol{\Omega}{}^{1}(\Sigma) \rightarrow \widehat{C}^{\infty}(\Sigma)_{\mu}, \quad
	J(\ol{\gamma}_{f}) = \textstyle \frac{1}{\pi}\Big(\pi^*\ol{\gamma}_{f} - \textstyle{\frac{1}{2}}\ol{f\theta}\Big).
\]
Indeed, note that $J$ covers the projections to $C^{\infty}(\Sigma)$
because
\[
d J(\ol{\gamma}_{f}) = \textstyle \frac{1}{\pi}\Big(\pi^* d\gamma_{f} - \textstyle{\frac{1}{2}}df \wedge \theta - \textstyle{\frac{1}{2}} f \wedge d\theta\Big)
= \frac{1}{2\pi} \big(fd\theta - df \wedge \theta\big).
\]
To see that $J$ is a Lie algebra homomorphism, one then uses that the Lie bracket on both sides factors through the projection to the Poisson algebra,
\[
[J(\ol{\gamma}_{f}), J(\ol{\gamma}_{g})] = \textstyle \frac{1}{2\pi}\big( \ol{fdg - gdf} - \ol{\{f,g\}\theta}\big) = J([\ol{\gamma}_{f}, \ol{\gamma}_{g}]).
\]
On the centre $H^1(\Sigma)$, the map $J$ coincides with the pullback 
$\pi^* \colon H^1_{\dR}(\Sigma) \rightarrow H^1_{\dR}(P)$, which is an isomorphism by the Gysin sequence.

The exact diagram \eqref{eq:anothernicediagram} can therefore be extended as follows:
	\begin{equation}\label{eq:Owateenmooidiagram}
		\begin{tikzcd}[row sep = scriptsize, column sep = large]
			 & &0 \ar[d] & 0\ar[d]\\
			&0\ar[d] &\ar[l]H^1_{\dR}(P)\ar[d]& H^1_{\dR}(\Sigma)\ar["\pi^*"', "\sim", l]\ar[d] & 0 \ar[l]\\
			0 &\ar[l]\frac{1}{2\pi} \R \ol{\theta} \ar[d]&\widehat{C^{\infty}}(\Sigma)_{\mu}\ar[d]\ar[l, "Q"']& \ol{\Omega}{}^{1}(\Sigma) \ar[l, "J"']\ar[d] & 0 \ar[l]\\
			0 &\ar[l]\R \ar[d] &C^{\infty}(\Sigma) \ar[d]\ar[l, "q"'] &C^{\infty}_{0}(\Sigma) \ar[l, "\iota"']\ar[d]& \ar[l]0.\\
			& 0 &0 & 0
		\end{tikzcd}
	\end{equation}
The top and the bottom row are split exact sequences of Fr\'echet--Lie algebras,
and we have identified $\X_{\ham}(\Sigma)$ with $C^{\infty}_{0}(\Sigma)$.
Note that $\frac{1}{2\pi} \R\ol{\theta}$ is a 1-dimensional central subalgebra of $\widehat{C^{\infty}}(\Sigma)_{\mu}$
which is not contained in the image of $J$. Since $J$ restricts to an isomorphism $H^1_{\dR}(\Sigma) \rightarrow H^1_{\dR}(P)$ 
and induces the codimension 1 inclusion $\iota \colon C^{\infty}_{0}(\Sigma) \hookrightarrow C^{\infty}(\Sigma)$, 
we have an isomorphism 
\begin{equation}\label{eq:Q}
\widehat{C^{\infty}}(\Sigma)_{\mu}/ J(\ol{\Omega}^1(\Sigma))\simeq \textstyle\frac{1}{2\pi} \R\ol{\theta}.
\end{equation}
Explicitly, the quotient map $Q \colon \widehat{C^{\infty}}(\Sigma)_{\mu} \rightarrow \frac{1}{2\pi} \R\ol{\theta}$ is realized by
\begin{equation}\label{eq:Qint}
Q(\ol{\psi}_{f}) = \frac{1}{\mathrm{vol}(\Sigma)}\Big(\int_{P}\psi_{f}\wedge d\theta \Big)\ol{\theta},
\end{equation}
and it covers the quotient map
\[
	q \colon C^{\infty}(\Sigma) \rightarrow \R, \quad q(f) = \frac{1}{\mathrm{vol}(\Sigma)}\int_{\Sigma} f\sigma.
\]
In particular, \emph{all} three rows are split exact.

A central extension of Fr\'echet--Lie algebras is called \emph{universal} if 
it is continuous and linearly split, and if 
every other continuous, linearly split central extension is covered by it in a unique fashion.
It is called \emph{versal} if the factorization exists, but is not necessarily unique.

Since the Poisson algebra $C^{\infty}(\Sigma)$ is not perfect, it does not have a universal extension. 
However, it is not hard to see from the above discussion that the Lie algebra extension that comes from 
$\widehat{\Aut}_{0}(P,\theta) \rightarrow \Aut_{0}(P,\theta)$ is versal.

\begin{proposition}
The central Lie algebra extension  
\begin{equation}\label{eq:versal}
H^1_{\dR}(P) \rightarrow \widehat{C^{\infty}}(\Sigma)_{\mu} \rightarrow C^{\infty}(\Sigma)
\end{equation}
corresponding to the central Lie group extension 
\begin{equation}
	H^1(P,\U(1)) \rightarrow \widehat{\Aut}_{0}(P,\theta)_{\cG} \rightarrow \Aut_{0}(P,\theta) 
\end{equation}
is versal.
Every continuous, linearly split central extension $\fz \rightarrow \fg^{\sharp} \rightarrow C^{\infty}(\Sigma)$
is covered by \eqref{eq:versal}, and the space of coverings is linearly isomorphic to $\fz$.
\end{proposition}
\begin{proof}
Choose a continuous linear splitting $s \colon C^{\infty}(\Sigma) \rightarrow \fg^{\sharp}$, and construct the corresponding cocycle
\[\psi \colon C^{\infty}(\Sigma) \times C^{\infty}(\Sigma) \rightarrow \fz, \quad \psi(f,g) = [s(f),s(g)] - s(\{f,g\}).\]
Since $C_{0}(\Sigma)$ is perfect, the cocycle $\psi$ vanishes on $\R \times C^{\infty}_{0}(\Sigma)$,
\[
\psi(1, \{f,g\}) = \psi(f, \{1,g\}) + \psi(g, \{f,1\}) = 0.
\]
Since $\R$ is 1-dimensional, it vanishes on $\R \times \R$ as well, so the skew-symmetric cocycle $\psi$ factors through 
$C^{\infty}_0(\Sigma) \times C^{\infty}_0(\Sigma)$. We conclude that $s(1)$ is central in $\fg^{\sharp}$.
Since \[H^1_{\dR}(\Sigma)\rightarrow \ol{\Omega}{}^1(\Sigma) \rightarrow C^{\infty}_0(\Sigma)\] is universal \cite{JV15},
there exists a unique continuous Lie algebra homomorphism $H \colon \ol{\Omega}{}^1(\Sigma) \rightarrow \fg^{\sharp}$ that covers the restriction of 
$\fg^{\sharp}$ to $C^{\infty}_0(\Sigma)$. 
Since the splitting $Q$ from \eqref{eq:Qint} is continuous, we have $\widehat{C^{\infty}}(\Sigma) = \ol{\Omega}{}^1(\Sigma) \oplus \frac{1}{2\pi}\R\ol{\theta}$
as a direct sum of Fr\'echet--Lie algebras.
Since $s(1)$ is central, we can complete $H$ to a covering 
$\widehat{C^{\infty}}(\Sigma) \rightarrow \fg^{\sharp}$ of central extensions 
by setting $H(\frac{1}{2\pi}\ol{\theta}) = s(1) + z$ for an arbitrary element $z\in \fz$.
\end{proof}

\section{The universal central extension of \texorpdfstring{$\mathrm{Ham}(\Sigma)$}{Ham(Sigma)}}\label{sec:UniversalHamilton}

In this section, we start with the same data as in Section \ref{sec:quant}:
our 2-dimensional manifold $\Sigma$ is a compact, connected surface with integral symplectic form $\sigma \in \Omega^{2}(\Sigma)$, the map $\pi: P \rightarrow \Sigma$ is a principal $\U(1)$-bundle over $\Sigma$, and $\theta \in \Omega^{1}(P,\R)$ is a connection of curvature $\sigma$.

In contrast to the Lie algebra $C^{\infty}(\Sigma)$ of the quantomorphism group $\Aut_{0}(P,\theta)$, the Lie algebra 
$\X_{\ham}(\Sigma)$ of $\Ham(\Sigma)$ is perfect, and therefore possesses a universal central extension.
It is the Lie algebra $\ol{\Omega}{}^1(\Sigma)$ from \eqref{eq:UnivCE}, with the Lie bracket $[\ol{\gamma_{f}}, \ol{\gamma_{g}}] = {\textstyle\frac{1}{2}}\ol{fdg - gdf}$.
We construct a universal central extension of the Fr\'echet--Lie group $\Ham(\Sigma)$ whose Lie algebra is $\ol{\Omega}{}^1(\Sigma)$.
In order to do so, we will use the following Recognition Theorem of Neeb, which cleanly separates the topological from the infinitesimal aspects
of the problem.

\begin{theorem}[Recognition Theorem]\label{RecognitionTheorem}
Let $\Ab \rightarrow \widehat{G} \rightarrow G$ be a central extension of Fr\'echet--Lie groups with finite-dimensional centre $\Ab$.
Suppose that: 
\begin{itemize} 
\item[1)] The Lie algebra $\fg$ is perfect, and the derived Lie algebra extension $\fz \rightarrow \widehat{\fg} \rightarrow \fg$ is universal;
\item[2)] $\widehat{G}$ is simply connected;
\item[3)] $\pi_1(G)$ is contained in the commutator group of the simply connected cover.
\end{itemize}
Then $\Ab \rightarrow \widehat{G} \rightarrow G$ is universal for central extensions of $G$ by regular abelian Lie groups modelled on 
complete locally convex spaces.
\end{theorem}
\begin{proof}
This follows from \cite[Corollary 2.13, Lemma 4.5, Proposition 4.13]{Neeb2002CE}.
\end{proof}

\subsection{A Lie group extension whose Lie algebra extension is universal}

We first construct a central extension of $\Ham(\Sigma)$ whose Lie algebra is the universal Lie algebra 
extension $\ol{\Omega}{}^{1}(\Sigma)$.
By Corollary~\ref{prop:smoothfunctor}, the embedded Lie subgroup $\U(1) \subseteq  \Aut_{0}(P,\theta)$ of vertical 
automorphisms
is covered by the embedded Lie subgroup $\widehat{\U(1)}_{\cG} \subseteq \widehat{\Aut}_{0}(P,\theta)_{\cG}$.
\begin{proposition}
If $P$ is a generator of the Picard group of $\Sigma$, then the central extension 
\[
	H^1(P,\U(1)) \rightarrow \widehat{\U(1)}_{\cG} \rightarrow \U(1)
\]
of the group $\U(1)$ of vertical automorphisms
is split. The group of splittings $s \colon \U(1) \rightarrow \widehat{\U(1)}_{\cG}$ is isomorphic to $\Hom(\U(1), \U(1)^{2g}) \simeq \Z^{2g}$.
\end{proposition}
\begin{proof}
The homological Gysin sequence for the circle bundle $P\rightarrow \Sigma$ yields an exact sequence
\[
	H_2(\Sigma, \Z) \stackrel{\frown \chi\;\;}{\longrightarrow} H_0(\Sigma, \Z) \longrightarrow H_1(P,\Z) \stackrel{\pi_*}{\longrightarrow} H_1(\Sigma,\Z) \rightarrow 0,
\]
where $\chi$ is the Euler class of $P \rightarrow \Sigma$. Since $\Sigma$ is closed and two dimensional, we have 
\[H_0(\Sigma,\Z) /(\chi \frown H_2(\Sigma,\Z)) = \Z/n\Z,\] with 
$n$ the Euler class of $P\rightarrow \Sigma$.
This yields an exact sequence 
\[
	H^1(\Sigma, \U(1)) \stackrel{\pi^*}{\longrightarrow} H^1(P,\U(1)) \longrightarrow C_n
\]
in cohomology, where $C_n \subset \U(1)$ is the cyclic group of order $n$.
If $\Sigma$ is a surface of genus $g$, we thus have an exact sequence
\[
	\U(1)^{2g} \rightarrow H^1(P,\U(1)) \rightarrow C_n.
\]
Since $n=1$ for every generator of the Picard group, we find in particular that $H^1(P,\U(1)) \simeq \U(1)^{2g}$.

A result of Shapiro \cite{Shapiro1949} says that for a compact, connected Lie group $K$, every central extension by $\U(1)^{2g}$ comes from a 
homomorphism $\pi_0(K) \rightarrow \U(1)^{2g}$,
\[
	\U(1)^{2g} \rightarrow \widetilde{K}\times_{\pi_0(K)} \U(1)^{2g} \rightarrow K.
\]
The extension admits a splitting if and only if the homomorphism $\pi_0(K) \rightarrow \U(1)^{2g}$ extends to 
the universal cover $\widetilde{K}$, and the group of splittings is 
isomorphic to $\Hom(K, \U(1)^{2g})$. For $K = \U(1)$, this yields the required result.
\end{proof}

So $s(\U(1))$ is an embedded subgroup of $\widehat{\U(1)}_{\cG}$, and therefore of $\widehat{\Aut}_{0}(P,\theta)_{\cG}$
because $\widehat{\U(1)}_{\cG} \subset \widehat{\Aut}_{0}(P,\theta)_{\cG}$ is an embedded subgroup by Corollary~\ref{prop:smoothfunctor}.
 It corresponds to the 1-dimensional abelian Lie algebra spanned by $\frac{1}{2\pi}\ol{\theta} + \pi^*[\alpha]$
for a class $[\alpha] \in H^1_{\dR}(\Sigma, \R)$.
Note that the diagram \eqref{eq:Owateenmooidiagram} remains commutative if we replace $\frac{1}{2\pi}\R\ol{\theta}$ in the left upper corner 
by $(\frac{1}{2\pi}\ol{\theta} + [\alpha])\R$. (The formula for the quotient map changes slightly due to the different choice of complement.)

\begin{proposition}\label{Prop:ConstructieExtensie}
If $P$ is a generator of the Picard group of $\Sigma$, and if $s\colon  \U(1) \rightarrow \widehat{\U(1)}_{\cG}$ 
splits the central extension $\widehat{\U(1)}_{\cG} \rightarrow \U(1)$, then 
the Lie algebra extension corresponding to
\[H^1(\Sigma, \U(1)) \rightarrow\widehat{\Aut}_{0}(P)/s(U(1)) \rightarrow \Ham(\Sigma)\]
 is the universal central extension 
 \[
 	H^1_{\dR}(\Sigma) \rightarrow \ol{\Omega}{}^{1}(\Sigma) \rightarrow \X_{\ham}(\Sigma)
 \]
 of the Lie algebra of hamiltonian vector fields.
\end{proposition}

\subsection{The topology of \texorpdfstring{$\Ham(\Sigma)$}{Ham(Sigma)}}

In order to apply the recognition theorem, we need to understand the fundamental group of $\Ham(\Sigma)$.
The first step is the following classical result by Smale, Earle--Eells and Gramain on the topology of diffeomorphism groups of closed 2-manifolds.

\begin{theorem}\label{Thm:TopologyDiffeoSurface}
Let $\Sigma$ be a closed, orientable surface.
The homotopy type of the connected diffeomorphism group $\Diff_{0}(\Sigma)$ is that of 
$\mathrm{SO}(3)$ for the sphere $\Sigma = S^2$, that of $T^2$ for the torus $\Sigma = T^2$, and trivial for all other closed, connected 
Riemann surfaces.
\end{theorem}

\begin{proof}
The case of the sphere is a classical result of Smale \cite[Theorem~A]{Smale1959}, the case of the torus and the case of genus $>1$
was obtained by Earle and Eells using analytic methods \cite[p.~21]{EarleEells1969}. An alternative proof using purely differential topological methods
was later found by Gramain \cite[Th\'eor\`eme~1]{Gramain1973}.
\end{proof}

We consider the flux homomorphism $\mathrm{Flux}_{\sigma}: \Diff(\Sigma, \sigma)_{0} \rightarrow J^{1}(\Sigma),\phi \mapsto \phi^{*}h - h$, where $h \in \widehat{H}^{1}(\Sigma,\U(1))$ is an arbitrary differential character with curvature $\sigma$ (see \cite[Section 3]{DJNV21}).
We then let $\Ham(\Sigma)_{+}$ be the kernel of $\mathrm{Flux}_{\sigma}$, and observe that it is not necessarily connected.
Since $\Ham(\Sigma)$ is the identity component of $\Ham(\Sigma)_{+}$, it has the same homotopy groups in every degree except zero.

\begin{corollary}\label{Cor:HomotopyType}
Let $S^2$ be the sphere, $T^2$ the torus, and $\Sigma_{g}$ a closed, orientable surface of genus $g>1$.
The homotopy type of $\Ham(S^2)_{+}$ is that of $\mathrm{SO}(3)$, the homotopy type of $\Ham(T^2)_{+}$ is trivial, and the homotopy type of $\Ham(\Sigma_g)_{+}$ is that of $\Z^{2g}$.
\end{corollary}
\begin{proof}
By \cite[Theorem 2.5.3]{Hamilton82}, the Fr\'echet--Lie group $\Diff(\Sigma)$ acts smoothly and transitively on the convex space $\mathcal{M}_V(\Sigma)$ of smooth positive measures of total measure $V = \int_{\Sigma}\sigma$, and the orbit map $\phi \mapsto \phi \cdot \sigma$ makes $\Diff(\Sigma) \rightarrow \mathcal{M}_V(\Sigma)$ into a smooth principal $\Diff(\Sigma,\sigma)$-bundle.
As $\mathcal{M}_V(\Sigma)$ is contractible, the homotopy type of $\Diff(\Sigma,\sigma)$ is the same as that of $\Diff(\Sigma)$.

Recall that $\Ham(\Sigma)_{+}$ is the kernel of the flux homomorphism. The flux homomorphism is surjective onto the Jacobian torus $J^1(\Sigma) = H^1(\Sigma, \R)/H^1(\Sigma, \R)_{\Z}$ because it is surjective at the Lie algebra level, so we have an exact sequence 
\begin{equation}\label{eq:HamDiffJac}
	1\rightarrow \Ham(\Sigma)_{+} \rightarrow \Diff(\Sigma,\sigma)_0 \rightarrow J^1(\Sigma) \rightarrow 1
\end{equation}
of Fr\'echet--Lie groups.
Since $J^1(\Sigma) \simeq T^{2g}$ for a surface of genus $g$, we have $\pi_1(J^1(\Sigma)) = \Z^{2g}$ and all other homotopy groups vanish.
The long exact sequence of homotopy groups for \eqref{eq:HamDiffJac} therefore yields 
$\pi_i(\Ham(\Sigma)_{+}) \simeq \pi_i(\Diff(\Sigma,\sigma)_{0})$ for $i>1$, and 
\[
1 \rightarrow \pi_1(\Ham(\Sigma)_{+}) \rightarrow \pi_1(\Diff(\Sigma,\sigma)_{0}) \rightarrow \Z^{2g} \rightarrow \pi_0(\Ham(\Sigma)_{+}) \rightarrow 1
\]
for the fundamental groups. 
Combined with Theorem~\ref{Thm:TopologyDiffeoSurface}, this yields the required information.
For $\Sigma = S^2$, we conclude that the homotopy type of $\Ham(S^2)_{+}$ is that of $\mathrm{SO}(3)$.
For $\Sigma = \Sigma_g$ of genus $g>1$, we conclude that $\pi_0(\Ham(\Sigma)_{+}) \simeq \Z^{2g}$, and that 
$\pi_i(\Ham(\Sigma)_{+})$ is trivial for $i\geq 1$.
Finally, for the 2-torus $\Sigma = T^2$, we have $\Diff_{0}(T^2, \sigma) = \mathrm{Ham}(T^2)_{+} \times T^2$ as a product of Lie groups, and the flux homomorphism 
is the projection onto the second factor. Since the exact sequence for the flux homomorphism splits, we find $\pi_0(\Ham(T^2)_{+}) = 0$, $\pi_1(\Ham(T^2)) = 0$, 
and $\pi_i(\Ham(T^2)) = 0$ for $i>1$ because this is so for $\Diff_0(T^2)$.
\end{proof}

For $\Sigma = S^2$, the Lie algebra $\X_{\ham}(S^2)$ of $\Ham(S^2)$ is centrally closed, as 
the continuous second Lie algebra cohomology $H^2(\X_{\ham}(S^2), \R) \simeq H^1_{\dR}(S^2)$ 
vanishes \cite{JV15}. In this case, the universal central extension is 
just the simply connected cover, with centre $\pi_1(\Ham(S^2)) = \Z/2\Z$.

\subsection{The universal central extension of \texorpdfstring{$\Ham(\Sigma)$}{Ham(Sigma)}}

For a closed surface $\Sigma$ of genus $g\geq 1$, let $G :=\widehat{\Aut}_{0}(P)/\sigma(U(1))$
be the central extension of $\Ham(\Sigma)$ constructed in
Proposition~\ref{Prop:ConstructieExtensie}, and let $\widetilde{G}$ be its simply connected cover.
Since $\Ham(\Sigma)$ is simply connected,
the kernel of the lifted map $\widetilde{G} \rightarrow \Ham(\Sigma)$ is central in $G$.
The central extension 
\[
1 \rightarrow H^1(\Sigma, \U(1))\rightarrow G \rightarrow  \Ham(\Sigma) \rightarrow 1
\]
therefore gives rise to an extension of Fr\'echet--Lie groups
\[
1 \rightarrow H^1(\Sigma, \R)\rightarrow \widetilde{G} \rightarrow  \Ham(\Sigma) \rightarrow 1
\]
which is still central, and which corresponds to the same Lie algebra extension.
As $\pi_1(\widetilde{G})$ and $\pi_1(\Ham(\Sigma))$ are trivial, and since the Lie algebra extension is universal by \cite[Theorem 5.6]{JV15}, 
we conclude from the Recognition Theorem~\ref{RecognitionTheorem} that $\widetilde{G} \rightarrow \Ham(\Sigma)$ is universal. 
\begin{theorem}
For $\Sigma = S^2$, the universal central extension of $\Ham(\Sigma)$ is the simply connected cover 
\[
1 \rightarrow \Z/2\Z \rightarrow \widetilde{\Ham}(S^2) \rightarrow \Ham(S^2) \rightarrow 1.
\]
For a 2-dimensional, closed, orientable surface $\Sigma$ of genus $g\geq 1$, the central extension 
\[
1 \rightarrow H^1(\Sigma, \R)\rightarrow \widetilde{G} \rightarrow  \Ham(\Sigma) \rightarrow 1
\]
is universal for regular abelian Lie groups modelled on locally convex spaces. The corresponding Lie algebra extension is 
$ 0 \rightarrow H^1_{\dR}(\Sigma)\rightarrow \ol{\Omega}{}^1(\Sigma) \rightarrow \X_{\ham}(\Sigma) \rightarrow 0$.
\end{theorem}


\begin{thebibliography}{GBMW83}
\small
	\bibitem[ALD74]{ALDM74}
	A.~Avez, A.~Lichnerowicz, and A.~Diaz-Miranda,
		{\it Sur l'alg\`ebre des automorphismes infinit\'esimaux d'une
              vari\'et\'e symplectique}
           J. Differential Geometry
    	{\bf 9} (1974), 1--40


	\bibitem[Ba78]{Banyaga1978}
	A.~Banyaga,
	\emph{Sur la structure du groupe des diff\'eomorphismes qui pr\'eservent une forme symplectique},
	Comment. Math. Helv. 53:2 (1978), 174--227


	\bibitem[BB14]{BB14} C.~B\"ar and C.~Becker, ``Differential Characters''
	Lecture Notes in Mathematics {\bf 2112}, Springer, Cham, 2014

	\bibitem[Br93]{B93} J.-L.~Brylinski, ``Loop spaces, characteristic classes and geometric quantization'', vol. 107 of {Progr.~Math.},
	Birkh\"auser, 1993.

	\bibitem[BM94]{BM94} J.-L.~Brylinski and D.A.~McLaughlin, {\it The geometry of degree four 
	characteristic classes and line bundles over loop space I},
	Duke Math. J.  {\bf 75:3}, 603--638, 1994.

	\bibitem[BMS21]{BMS21} S.~Bunk, L.~M\"uller, R.~Szabo, {\it Smooth 2-group extensions and symmetries of bundle gerbes}, 
	Comm.~Math.~Phys. {\bf 384:3} (2021), 1829--1911.
	
	\bibitem[BS23]{BS23}S.~Bunk and C. S. Shahbazi, 
	 {\it Higher geometric structures on manifolds and the gauge theory of Deligne cohomology}, 
	 {\tt arxiv:2304.06633}

	\bibitem[Ca70]{Calabi1970}
	E.~Calabi,
	\emph{On the group of automorphisms of a symplectic manifold},
	in `Problems in analysis' ({L}ectures at the {S}ymposium in honor of
	{S}alomon {B}ochner, {P}rinceton {U}niv., {P}rinceton,
	{N}.{J}., 1969), pp. 1--26, Princeton Univ. Press, 1970

	\bibitem[CF64]{ConnerFloyd1964}
	P.~E.~Connor and E.~E.~Floyd,
	``Differentiable Periodic Maps'', Ergebnisse der Mathematik und ihrer Grenzgebiete, Band 33, 
	Springer Berlin Heidelberg, 1964

	\bibitem[DJNV21]{DJNV21} T.~Diez, B.~Janssens, K.H.~Neeb and C.~Vizman,
	{\it Central extensions of Lie groups preserving a differential form},
	Int.~Mat.~Res.~Notices {\bf 2021:5}, 3794--3821, 2021. 
	
	\bibitem[DJNV20]{DJNV20} 
	T.~Diez, B.~Janssens, K.H.~Neeb and C.~Vizman,
	{\it Induced differential characters on nonlinear Grassmannians}, 
	arXiv:2009.02556. (To appear in Ann. Inst. Fourier.)
	
	
	\bibitem[EE69]{EarleEells1969}
	C.~J.~Earle and J.~Eells,
	{\it A fibre bundle description of Teichm\"uller theory},
	J.~Diff.~Geom. {\bf 3}, 1969, 19--43.
	
	\bibitem[FRS14]{FRS14} D.~Fiorenza, C.~Rogers and U.~Schreiber, {\it $L_{\infty}$-algebras of local observables from higher prequantum bundles}, Homol.\ Homotopy Appl.\ {\bf 16:2} (2014), 107--142.


	\bibitem[Fr81]{Fr81} A.~Fr\"olicher, {\it Applications lisses entre espaces et vari\'et\'es de Fr\'echet}, 
	C. R. Acad. Sci. Paris S\'er. I Math. {\bf 293:2} (1981), 125--127.
	
	\bibitem[Gr73]{Gramain1973}
	A.~Gramain,
	{\it Le type d'homotopie du groupe des diff\'eomorphismes d'une surface compacte},
	Ann.~scient.~\'Ec.~Norm.~Sup., 4e s\'erie, tome 6, 53--66.
	
		\bibitem[HV04]{HV04}
	S.~Haller, and C.~Vizman, {\it Non--linear Grassmannians as coadjoint orbits},
	Math. Ann. \textbf{329:3}  (2004), 771--785.

	\bibitem[Ha82]{Hamilton82}
	R.~S.~Hamilton,
	{\it The inverse function theorem of Nash and Moser},
	Bull.~Amer.~Math.~Soc. \textbf{7:1} (1982), 65--222.
	
	\bibitem[Is96]{Ismagilov1996} R.~S.~Ismagilov,
	{\it Representations of infinite-dimen\-sional groups,}
	Translations of Mathematical Monographs {\bf 152},
	American Math. Soc., Providence, RI, 1996
	

	\bibitem[JRV23]{LeoCornelia}
	B.~Janssens, L.~Ryvkin, and C.~Vizman,
	{\it Universal central extension of the Lie algebra of exact divergence free vector fields},
	in progress.


	\bibitem[JV15]{JV15} B.~Janssens and C.~Vizman, {\it Universal central extension of the Lie algebra of Hamiltonian vector fields}, Int.~Math.~Res.~Notices {\bf 2016:16} (2015), 4996--5047.
	
	\bibitem[JV18]{JV18}
	B.~Janssens and C.~Vizman,
	{\it Integrability of central extensions of the Poisson Lie algebra via prequantization}
	 J. Symplectic Geom. {\bf 16:5} (2018), 1351--1375.
	
	\bibitem[Ko70]{Ko70} B.~Kostant, {\it Quantization and unitary representations}, Lectures in modern analysis and appl. 3, Lect. Notes Math 
	{\bf 170}, 87--208, 1970.
	

	\bibitem[Lo92]{Lo92} M.~Losik, 
	{\it Fr\'echet manifolds as diffeological spaces}, 
	Izv. Vyssh. Uchebn. Zaved. Mat. {\bf 36:5} (1992), 36--42.

	\bibitem[Mu96]{Mu96}
	M.~Murray,
	{\it Bundle Gerbes},
	J.~Lond.~Math.~Soc.~{\bf 54} (1996), 403-416
	
	\bibitem[Ne02a]{Neeb2002CE} K.-H.~Neeb,
	{\it Central extensions of infinite-dimensional Lie groups},
	Annales de l'Inst.\ Fourier {\bf 52} (2002), 1365--1442

	\bibitem[Ne02b]{Neeb2002UCE}
	K.-H.~Neeb,
	{\it Universal central extensions of Lie groups}
	Acta Appl. Math. {\bf 73} (2002), 175--219.
	
	\bibitem[Ne04]{Ne04} K.-H.~Neeb,
	{\it Abelian Extensions of Infinite-Dimensional Lie Groups,}
	Travaux math\'ematiques {\bf 15} (2004) 69--194

	\bibitem[Ne05]{Ne05} K.-H.~Neeb,
	{\it Lie algebra extensions and higher order cocycles},
	J. Geom. Symmetry Phys. {\bf 4} (2005), 1--27

	\bibitem[NV03]{NV03}
	K.-H.~Neeb, and C.~Vizman,
	{\it Flux homomorphism and principal bundles over infinite dimensional manifolds}
	Monatsh. Math. {\bf 139} (2003), 309--333
		
	\bibitem[Om74]{Om74}
   	H.\ Omori,
   	{\it Infinite dimensional Lie transformation groups},
   	Lecture Notes in Mathematics, Vol. {\bf 427},
   	Springer-Verlag, Berlin-New York,
   	1974

	\bibitem[RS81]{RatiuSchmid1981}
	T.~Ratiu and R.~Schmid, 
	{\it The differentiable structure of three remarkable diffeomorphism groups},
	Math. Z. {\bf 177} (1981), 81--100

	\bibitem[Ro95]{Ro95}
	C.~Roger, {\it Extensions centrales d'alg\`ebres et de groupes de Lie de dimension
	infinie,alg\`ebre de Virasoro et g\'en\'eralisations},
	Rep. Math. Phys. {\bf 35} (1995), 225-266.

	\bibitem[Sh49]{Shapiro1949}
	A.~Shapiro,
	{\it Group extensions of compact Lie groups},
	Ann. Math. {\bf 50:3} (1949), 581--586

	\bibitem[Sm59]{Smale1959}
	S.~Smale,
	{\it Diffeomorphisms of the $2$-sphere},
	Proc. Amer. Math. Soc. {\bf 10} (1959), 621--626

	\bibitem[Th54]{Thom1954}
	R.~Thom, {\it Quelques propri\'et\'es globales des vari\'et\'es diff\'erentiables},
	Comment. Math. Helv. {\bf 28} (1954), 17--86


	\bibitem[Wa07a]{WaldorfPHD}
	K.~Waldorf,
	{\it Algebraic structures for Bundle Gerbes and the Wess-Zumino Term in Conformal Field Theory},
	PhD Thesis, Universit\"at Hamburg, 2007

	\bibitem[Wa07b]{Wa07}
	K.~Waldorf,
	{\it More morphisms between bundle gerbes},
	Theory and Applications of Categories {\bf 18:9} (2007), 240--273

	\bibitem[Wa10]{Wa10}
	K.~Waldorf,
	{\it Multiplicative Bundle Gerbes with Connection},
	Differential Geom. Appl. {\bf 28:3} (2010), 313--340
	
	\bibitem[Wa16]{Wa16}
	K.~Waldorf,
	{\it Transgression to Loop spaces and its inverse, II: Gerbes and fusion bundles with connection},
	Asian J. Math {\bf 20:1}, 59--116.

\end{thebibliography}
\end{document}